\documentclass{article}
\usepackage{graphicx}
\usepackage{subfigure}
\usepackage{amsfonts}
\usepackage{amsmath}
\usepackage{amssymb}
\usepackage{url}
\usepackage{fancyhdr}
\usepackage{indentfirst}
\usepackage{enumerate}
\usepackage{cancel}
\usepackage{dsfont}
  \usepackage[colorlinks=true,citecolor=blue]{hyperref}
\usepackage{amsthm}
\usepackage{color}
\usepackage{natbib}
\usepackage{comment}
\usepackage{capt-of}
\usepackage{multirow}
\usepackage[T1]{fontenc}
\usepackage[utf8]{inputenc}
\def\d{\mathrm{d}}

\newcommand{\E}{\mathbb{E}}

\newcommand{\R}{\mathbb{R}}

\newcommand{\p}{\mathbb{P}}

\newcommand{\id}{\mathds{1}}

\renewcommand{\ge}{\geqslant}
\renewcommand{\le}{\leqslant}
\renewcommand{\geq}{\geqslant}
\renewcommand{\leq}{\leqslant}
\renewcommand{\epsilon}{\varepsilon}
\newcommand{\esssup}{\mathrm{ess\mbox{-}sup}}
\newcommand{\essinf}{\mathrm{ess\mbox{-}inf}}

\renewcommand{\cdots}{\dots}

\theoremstyle{plain}
\newtheorem{theorem}{Theorem}
\newtheorem{corollary}{Corollary}
\newtheorem{lemma}{Lemma}
\newtheorem{proposition}{Proposition}
\theoremstyle{definition}

\newtheorem{example}{Example}

\theoremstyle{remark}

\newcommand{\cet}{\begin{center}}
\newcommand{\ecet}{\end{center}}

\usepackage{setspace}

\begin{document}

\title{Risk diversification for infinitely divisible distributions}

\author{
    Peng Liu\thanks{\scriptsize School of Mathematics, Statistics and Actuarial Science, University of Essex, 
    UK. Email: \texttt{peng.liu@essex.ac.uk}}
    \and
    Tiantian Mao\thanks{\scriptsize Department of Statistics and Finance, School of Management, University of Science and Technology of China, Hefei, China.  Email: \texttt{tmao@ustc.edu.cn}}
   }

\date{}
\maketitle



\begin{abstract}
 In this paper, we study the diversification properties of convex combinations of iid  infinitely divisible random variables. For L\'{e}vy processes with bounded variation sample paths, we characterize, in terms of subadditivity and concavity of the transformed L\'{e}vy tails,  L\'{e}vy processes that exhibit the non-diversification phenomenon or the reverse diversification order with respect to the majorization order uniformly over all time horizons. For general symmetric L\'{e}vy processes without a Gaussian component, we show that the symmetric 1-stable L\'{e}vy process is the only nontrivial process exhibiting  either phenomenon.  We further investigate convex combinations of components of multivariate infinitely divisible distributions, allowing for dependent and heterogeneous components, and characterize the L\'{e}vy measures of multidimensional  L\'{e}vy processes exhibiting the two adverse diversification phenomena uniformly over all time horizons. Explicit characterizations are obtained for  the multidimensional symmetric L\'{e}vy processes,  multidimensional $\alpha$-stable processes and  multidimensional compound Poisson processes. Finally, we show that the non-diversification phenomenon extends beyond L\'{e}vy processes to running maxima and integrals of increasing convex functionals of L\'{e}vy processes, while both adverse diversification phenomena are preserved for L\'{e}vy-driven stochastic integrals with  nonnegative deterministic kernels. Applications to ruin theory, storage processes and stochastic volatility are also discussed. 
\end{abstract}



\section{Introduction}

Diversification has been a  fundamental principle in finance, economics  and insurance since \cite{M52}, suggesting that portfolio diversification or risk sharing generally leads to risk  reduction.   However, the literature has shown that diversification may have adverse effects for risks with infinite mean; see, for example, \cite{F65} for an early result concerning stable distributions. Infinite-mean models arise in various contexts in insurance and finance. Examples include financial returns from some technological innovations in \cite{SHK01} and \cite{SV07}, and  losses from earthquakes in \cite{IJW09}, nuclear power accidents in \cite{HW12}, cyber risk in \cite{EW19} and operational risk in \cite{NEC06}, all of which have been modeled or numerically estimated as having infinite mean. We refer to \cite{IJW11}, \cite{CEW25} and \cite{CEW25b} for discussions of the implications of  infinite-mean models  in risk management.

 Stochastic dominance  is a fundamental tool widely applied in economic decision theory and finance for the analysis of the risk preferences of decision makers. The most commonly used stochastic orders are  first- and second-order stochastic dominance.  We refer to \cite{L16} for the applications in decision theory, and  \cite{MS02} and \cite{SS07} for the mathematical foundation of stochastic dominance. Majorization provides  a natural ordering for comparing  the diversity of the components of vectors.  Majorization order has been  applied in economics to analyze income inequality and its effects on  economic models; see e.g., \cite{II07} and \cite{MOA11}. It has also been used in finance and risk management to study portfolio and risk diversification; see, e.g., \cite{HL03}, \cite{IB09}, \cite{CHWZ25} and the references therein.

 The non-diversification phenomenon under the first-order stochastic dominance  and the reverse diversification order with respect to the majorization order were obtained in \cite{IB05} for portfolios consisting of iid positive one-sided $\alpha$-stable random variables. Later, \cite{IB09} investigated  the majorization order of portfolios consisting of iid random variables following convolutions of symmetric stable distributions, focusing on the Value-at-Risk for confidence levels above $1/2$, where symmetry plays a key role in the analysis.  More recently, a striking result was obtained in \cite{CEW25}, showing that a nontrivial  convex combination of iid Pareto risks with infinite mean strictly increases the risk in the sense of the first-order stochastic dominance. The result also holds true for some specific negatively dependent but identically distributed super-Pareto risks with infinite mean. The non-diversification phenomenon has been extended to iid random variables with other distributions including \cite{M25} for super-Cauchy distributions, \cite{ALO25} for inverted-subadditive distributions, and \cite{CS26} for super-Fr\'{e}chet distributions. Moreover, the reverse diversification order was obtained for iid random variables with Pareto distributions with infinite mean in   \cite{CHWZ25} and iid random variables with a class of distributions called inverted concavity  in \cite{CHSZ26}. 

 Note that $\alpha$-stable distributions, Pareto distributions and Fr\'{e}chet distributions with shape parameter not larger than one, and their convolutions all belong to the class of infinitely divisible distributions. Infinitely divisible distributions and L\'{e}vy processes are fundamental objects widely used in  risk theory, queueing and storage theory,  finance, and operations research.  We refer to the classical monographs \cite{B96} and \cite{S99}  for the definitions and characterization of infinitely divisible distributions and L\'{e}vy processes, and  \cite{AA10} for applications to insurance risk, \cite{DM15} for L\'{e}vy-driven queueing and storage models,  Chapter 7 of \cite{K14} for  first-passage problems, and \cite{BS01} and \cite{H09} for applications in finance. In this paper, rather than identifying additional specific classes of distributions exhibiting adverse diversification, we seek a structural characterization of such phenomena within the broad class of infinitely divisible distributions in terms of their L\'{e}vy measures. 

 In risk management and finance, many interesting risk- or profit-related quantities depend on the sample paths of the underlying stochastic processes such as running maxima representing the maximum loss over a time period, L\'{e}vy-driven stochastic integrals, and time integrals representing the cumulative path-dependent losses; see, e.g.,  Chapter XI of \cite{AA10}, \cite{BS01} and \cite{H09}. To establish stochastic dominance for such path-dependent quantities,  stochastic dominance at a single time point is generally insufficient. This motivates the study of  the stochastic dominance properties that hold  uniformly over all time horizons. Therefore, in this paper, we study the non-diversification phenomenon and the reverse diversification order for L\'{e}vy processes uniformly over all time horizons. We then apply these results to running maxima in risk theory, storage processes and  L\'{e}vy-driven Ornstein-Uhlenbeck stochastic volatility models. We highlight that the stochastic dominance uniformly over all time horizons is a stronger property than the one for a single time point, allowing for broader applications in risk management and finance.  Proposition \ref{ex:time-dependent-dominance} in Section \ref{Sec:example} further shows that this uniformity requirement is substantive: for a compound Poisson process,   the direction of the diversification effect may change with the time horizon. The diversification phenomena were recently established for iid compound Poisson processes in \cite{CHSZ26}, which, to the best of our knowledge, is the only existing study establishing those phenomena for stochastic processes. However, those phenomena for general L\'{e}vy processes and multidimensional L\'{e}vy processes allowing dependent and heterogeneous components remain largely unexplored. The objective of this paper is to characterize the structure of the L\'{e}vy measures under which the corresponding L\'{e}vy processes exhibit the non-diversification phenomenon or the reverse diversification order.

 The contributions of this paper are summarized below.
\begin{enumerate}

\item [1.] We give a necessary and sufficient characterization of adverse diversification uniformly over all time horizons for L\'{e}vy processes with bounded variation. The non-diversification phenomenon and the reverse diversification order correspond, respectively, to subadditivity and concavity of the transformed L\'{e}vy tail, and both require the absence of negative jumps. For univariate symmetric L\'{e}vy processes without a Gaussian component, the Cauchy process is the only nontrivial process exhibiting either phenomenon uniformly over all time horizons. We further construct a compound Poisson process for which the non-diversification phenomenon fails at a short horizon but holds at all sufficiently long horizons for each fixed nontrivial weight vector, showing that time-uniformity is a substantive restriction.

\item [2.] We identify the distinct roles of marginal heterogeneity and dependence in adverse diversification. For multivariate $\alpha$-stable processes with $0<\alpha<1$, nonnegative jumps, and a common drift across components, a component is stochastically dominated by every convex combination if and only if it has the smallest $\alpha$-th moment under the spectral measure. The reverse diversification order holds if and only if the process is exchangeable. These results identify the precise dependence symmetry required for reverse majorization comparisons beyond the iid setting.

\item [3.] We show that adverse diversification can persist for path-dependent risks. Under our L\'{e}vy-measure conditions, non-diversification is preserved for running maxima and integrals of increasing convex functionals, while both orders are preserved under nonnegative deterministic stochastic integration. The resulting comparisons show that diversification of a fixed exposure across independent copies can increase finite-horizon ruin probabilities and storage levels, and that mean reversion in L\'{e}vy-driven Ornstein-Uhlenbeck models need not restore diversification benefits.
\end{enumerate}

The organization of the paper is as follows. Notation and preliminaries are presented in Section \ref{sec:notation}. In Section \ref{Sec:iid}, we consider the iid case and characterize the L\'{e}vy measures of all  L\'{e}vy processes with bounded variation sample paths exhibiting the two adverse diversification phenomena. We also characterize symmetric L\'{e}vy processes exhibiting these phenomena.  Section \ref{Sec:multivariate} is devoted to the multivariate case, allowing for dependence and heterogeneous marginals. In particular, we characterize the L\'{e}vy measures of  multidimensional symmetric L\'{e}vy processes, multidimensional $\alpha$-stable processes and multidimensional compound Poisson processes.   In Section \ref{Sec:extension}, we extend the two adverse diversification phenomena beyond L\'{e}vy processes to several sample-path-dependent processes, including  running maxima of L\'{e}vy processes, L\'{e}vy-driven stochastic integrals with nonnegative deterministic kernels,  and integrals of increasing convex functionals of L\'{e}vy processes. Applications to ruin theory, storage processes and  stochastic volatility are also discussed. In Section \ref{Sec:example}, we construct a compound Poisson process for which the validity of  the stochastic dominance relation changes with the time horizon. Section \ref{Sec:conc} concludes the paper.   All the proofs, together with  one additional result on multivariate $\alpha$-stable distributions, are postponed to the Appendix. 


\section{Notation and Preliminaries}\label{sec:notation}

Let $(\Omega,\mathcal F, \p)$ be an atomless probability space. For a real-valued random variable $X$, let $F_X$ denote its cumulative distribution function under $\p$.  
For two random variables $X$ and $Y$, $X\leq_{\mathrm{st}} Y$ denotes that  $X$ is dominated by $Y$ in \emph{first-order} stochastic dominance, i.e., $F_X(x)\geq F_{Y}(x)$ for all $x\in \R$. We say $X\leq_{\mathrm{st}} Y$ \emph{strictly} if $X\leq_{\mathrm{st}} Y$ and $F_X\neq F_Y$; we say the stochastic dominance is  \emph{strongly strict} if $F_X(x)\geq F_{Y}(x)$ for all $x\in \R$, and  $F_X(x)>F_{Y}(x)$ for all $x\in (\essinf X,\esssup X)$.   We use $X\leq_{\mathrm{a.s.}} Y$ to denote $X\leq Y$ almost surely.  Throughout the paper, let $n\geq 2$.

 Let $X_1,\dots,X_n$ be iid random copies of $X$. 
We are interested in the conditions under which, for  $\boldsymbol\theta:=(\theta_1,\dots,\theta_n)\in \Delta_n:=\{(\theta_1,\dots,\theta_n): \theta_1+\dots+\theta_n=1, \theta_i\geq 0,~i\in [n]\}$, the following holds:
\begin{align}\label{Sup}
  X\leq_{\mathrm{st}} \theta_1X_1+\dots+\theta_nX_n.
\end{align}

The stochastic dominance property in \eqref{Sup} shows the non-diversification phenomenon, i.e., diversification may increase the risk in the sense of the first-order stochastic dominance. We are particularly  interested in the conditions on the distributions when \eqref{Sup} holds for all $\boldsymbol\theta\in\Delta_n$.

Moreover, we aim to explore the stochastic order for convex linear combinations of some iid random variables. 
For $\boldsymbol\theta=(\theta_1,\dots,\theta_n)\in\Delta_n$ and $\boldsymbol\eta=(\eta_1,\dots,\eta_n)\in\R_+^n$, we say $\boldsymbol\theta$ is dominated by $\boldsymbol\eta$ in the \emph{majorization order}, denoted by $\boldsymbol\theta \preceq\boldsymbol\eta$ , if $\sum_{i=1}^n\theta_i=\sum_{i=1}^n\eta_i$ and $\sum_{i=1}^m\theta_{(i)}\geq \sum_{i=1}^m\eta_{(i)}$ for $m\in [n]$, where $\theta_{(i)}$ and $\eta_{(i)}$ represent the $i$-th smallest order statistics of $\boldsymbol\theta$ and $\boldsymbol\eta$ respectively. We write $\boldsymbol\theta \prec\boldsymbol\eta$ if $\boldsymbol\theta \preceq\boldsymbol\eta$ and $\boldsymbol\theta$ is not a permutation of $\boldsymbol\eta$. In this paper, without loss of generality, we constrain $\boldsymbol\theta,\boldsymbol\eta\in \Delta_n$ when we study the majorization order. We refer to \cite{MOA11} for the definition, interpretation and application of the majorization order. 

For  $\boldsymbol\theta,\boldsymbol\eta\in \Delta_n$ satisfying $\boldsymbol\theta \preceq\boldsymbol\eta$  and iid copies of X, $X_1,\dots,X_n$, we next study the following stochastic order
\begin{align}\label{Sub}
    \sum_{i=1}^n\eta_iX_i\leq_{\mathrm{st}} \sum_{i=1}^n\theta_iX_i.
\end{align}

The stochastic dominance property in \eqref{Sub} shows the reverse diversification direction, i.e., a more diversified portfolio may have higher risk in the sense of the first-order stochastic dominance. Note that if \eqref{Sub} holds for all $\boldsymbol\theta,\boldsymbol\eta\in \Delta_n$ satisfying $\boldsymbol\theta \preceq\boldsymbol\eta$, then
\eqref{Sup} holds for all $\boldsymbol\theta\in\Delta_n$.

Note that the stochastic dominance in \eqref{Sup} or \eqref{Sub} holds for $X$ if and only if the reverse stochastic dominance  in \eqref{Sup} or \eqref{Sub} holds for  $-X$.  Hence, instead of studying the reverse stochastic dominance, we only focus on \eqref{Sup} and \eqref{Sub} in this paper.

Let $\mathbf X=(X_1,\dots,X_n)$ be a random vector with an infinitely divisible distribution $\mu_{\mathbf X}$. Then there exist an $n\times n$ non-negative definite matrix $A$, a measure $\nu$ on $\R^n$ satisfying $\nu(\{\mathbf 0\})=0$ and $\int_{\R^n} (\|\mathbf x\|^2\wedge 1)\nu(\d \mathbf x)<\infty$, and a $\boldsymbol\gamma\in\R^n$ such that the characteristic function $$\hat{\mu}_{\mathbf X}(\mathbf z)=\mathbb E(e^{i\langle\mathbf z, \mathbf X\rangle})=\exp\left[-\frac{1}{2}\langle \mathbf z,A\mathbf z\rangle+i\langle\boldsymbol{\gamma},\mathbf z\rangle+\int_{\R^n} \left(e^{i\langle\mathbf z,\mathbf x\rangle}-1-i\langle\mathbf z,\mathbf x\rangle\id_{D}(\mathbf x) \right)\nu(\d\mathbf x)\right],$$
where  $D:=\{\mathbf x\in\R^n: \|\mathbf x\|\leq 1\}$,   $i^2=-1$, $\langle \cdot,\cdot\rangle$ is the inner product and $\nu$ is called the \emph{L\'{e}vy measure.}
By convention, if $\int_{D}\|\mathbf x\|\nu(\d\mathbf x)<\infty$, then $\hat{\mu}_{\mathbf X}$ can be rewritten as
$$\hat{\mu}_{\mathbf X}(\mathbf z)=\exp\left[-\frac{1}{2}\langle \mathbf z,A\mathbf z\rangle+i\langle\boldsymbol{\gamma}_0,\mathbf z\rangle+\int_{\R^n} \left(e^{i\langle\mathbf z,\mathbf x\rangle}-1 \right)\nu(\d\mathbf x)\right].$$
The distribution $\mu_{\mathbf X}$ is uniquely determined by the generating triplet $(A,\nu, \boldsymbol \gamma)$. 
When $\int_{D}\|\mathbf x\|\nu(\d\mathbf x)<\infty$, it can equivalently
be characterized by $(A,\nu,\boldsymbol \gamma_0)$, where
$ 
\boldsymbol \gamma_0=\boldsymbol \gamma-\int_D \mathbf x\,\nu(d\mathbf x).
$ 

Let $\{X_t\}_{t\geq 0}$ be a L\'{e}vy process with generating triplet $(A,\nu, \boldsymbol \gamma)$; that is, $X_0=0$ a.s., $\{X_t\}_{t\geq 0}$ has stationary and independent increments, and  $X_1$ has a triplet  $(A,\nu, \boldsymbol \gamma)$. Since a L\'{e}vy process admits a c\`{a}dl\`{a}g modification, throughout the paper, we work on a version whose sample paths are right-continuous with left limits.  For more details on the definitions, properties, characterizations and applications of infinitely divisible distributions and L\'{e}vy processes, we refer to the classical monographs \cite{B96}, \cite{S99} and \cite{K14}.

In this paper, we study the structural properties of the L\'{e}vy measures  such that \eqref{Sup} or \eqref{Sub} holds for the corresponding L\'{e}vy processes uniformly for all time horizons.

Finally, we give some definitions and notation that will be used frequently in this paper.
Let $I$ be either $(-\infty,0)$ or $(0,\infty)$. We say a nonnegative measurable function $g$ is \emph{subadditive} (resp. \emph{strictly subadditive}) on $I$ if $g(x+y)\leq g(x)+g(y)$ (resp. $g(x+y)< g(x)+g(y)$ ) for all $x,y\in I$; we say a nonnegative measurable function $g$ is \emph{superadditive} (resp. \emph{strictly superadditive}) on $I$ if $g(x+y)\geq g(x)+g(y)$ (resp. $g(x+y)> g(x)+g(y)$ ) for all $x,y\in I$.  For $x\in\R$, let $x_+=\max(x,0)$ and $x_-=-\min(x,0)$.
\section{Stochastic dominance for infinitely divisible distributions}\label{Sec:iid}
In this section, we focus on the L\'{e}vy processes and obtain the characterization of the corresponding L\'{e}vy measures such that the stochastic dominance property in \eqref{Sup} or \eqref{Sub} holds uniformly for all time horizons.  We first establish the following result.

\begin{theorem}\label{th:main1} Let  $\{X_t\}_{t\geq 0}$ be a L\'{e}vy process with generating triplet $(0, \nu, \gamma_0)$ satisfying $\int_D|x|\nu(\d x)<\infty$. Then the following conclusions hold.
\begin{enumerate}[(i)]
    \item  The stochastic dominance in \eqref{Sup} holds for all $X_t$ with $t>0$ and all $\boldsymbol\theta\in \Delta_n$ if and only if $\nu((x^{-1},\infty))$ is subadditive on $(0,\infty)$ and $\nu((-\infty,0))=0$. Moreover, if $\nu((x^{-1},\infty))$ is strictly subadditive on $(0,\infty)$, then \eqref{Sup} holds strictly  whenever $t>0$ and  $\max_{i=1}^n\theta_i<1$.
    \item The stochastic dominance in \eqref{Sub} holds for all  $X_t$ with $t>0$  and  all  $\boldsymbol\theta,\boldsymbol\eta\in \Delta_n$ satisfying $\boldsymbol\theta \preceq\boldsymbol\eta$  if and only if $\nu((x^{-1},\infty))$ is concave on $(0,\infty)$ and $\nu((-\infty, 0))=0$. Moreover, if $\nu((x^{-1},\infty))$ is nonlinear and concave on $(0,\infty)$, then  \eqref{Sub} holds strictly whenever $t>0$ and $\boldsymbol\theta \prec\boldsymbol\eta$.
\end{enumerate}
\end{theorem}

Theorem \ref{th:main1} characterizes the L\'{e}vy measures for the validity of \eqref{Sup} and \eqref{Sub} for L\'{e}vy processes over all time horizons. In practice, the aggregated risk is modeled by a L\'{e}vy process. If the risk diversification cannot reduce the risk at any time, then the corresponding L\'{e}vy measure should satisfy the conditions in Theorem \ref{th:main1}. Conversely, if the L\'{e}vy measure satisfies the conditions in Theorem \ref{th:main1}, then risk diversification has a negative effect regardless of the time horizon considered. Hence, the conditions in Theorem \ref{th:main1} give rise to diversification properties uniformly over all time for infinitely divisible distributions. Moreover, Theorem \ref{th:main1} shows that strengthening  the subadditivity and concavity conditions on the transformed L\'{e}vy tails yields the strict stochastic dominance in \eqref{Sup} and \eqref{Sub}.

Comparing the conditions in (i) and (ii) of Theorem \ref{th:main1}, the concavity of $\nu((x^{-1},\infty))$ on $(0,\infty)$ is a stronger condition as it implies the subadditivity by noting that $\lim_{x\downarrow 0} \nu((x^{-1},\infty))=0$. Thus, \eqref{Sub} is essentially strictly stronger than \eqref{Sup}.

We  emphasize that in our assumption, negative jumps are allowed as we do not assume $\nu((-\infty,0))=0$. However, in our conclusions, \eqref{Sup} and \eqref{Sub} hold for the L\'{e}vy process only if $\nu((-\infty,0))=0$, meaning that negative jumps are not allowed. Here $\nu((-\infty,0))=0$ is not an assumption but is a part of the conclusion. From the proof of Theorem \ref{th:main1} in the Appendix, this is due to  the restriction $\int_{D}|x|\nu(\d x)<\infty$. In fact, the ``only if'' parts in Theorem \ref{th:main1} do not rely on the constraint on the L\'{e}vy measure, which is shown in the proof of Theorem \ref{th:main1}. It is displayed below as a separate result for later use.

\begin{proposition}\label{Prop:onlyif} Let  $\{X_t\}_{t\geq 0}$ be a L\'{e}vy process with generating triplet $(0, \nu, \gamma)$.  Then the following conclusions hold.
\begin{enumerate}[(i)]
    \item If \eqref{Sup} holds for all $X_t$ with $t>0$ and all $\boldsymbol\theta\in \Delta_n$, then  $\nu((x^{-1},\infty))$ is subadditive on $(0,\infty)$ and  $\nu((-\infty,x^{-1}))$ is superadditive on $(-\infty,0)$. 
    \item If \eqref{Sub} holds for all  $X_t$ with $t>0$  and  all  $\boldsymbol\theta,\boldsymbol\eta\in \Delta_n$ satisfying $\boldsymbol\theta \preceq\boldsymbol\eta$ , then $\nu((x^{-1},\infty))$ is concave on $(0,\infty)$ and $\nu((-\infty, x^{-1}))$ is convex on $(-\infty,0)$.
\end{enumerate}
\end{proposition}

Moreover, Theorem \ref{th:main1} also provides many examples for infinitely divisible random variables with infinite mean that do not satisfy \eqref{Sup} or \eqref{Sub}.
\begin{corollary}\label{Cor:diversification} Let $\nu$ be a L\'{e}vy measure satisfying $\int_D|x|\nu(\d x)<\infty$.  
The following two statements are equivalent.
\begin{enumerate}[(i)]
    \item  The L\'{e}vy measure satisfies $\nu((-\infty,0))>0$, or $\nu((x^{-1},\infty))$ is not subadditive on $(0,\infty)$;
    \item  The infinitely divisible random variable with  triplet $(0, t_0\nu,\gamma_0)$ does not satisfy \eqref{Sup} for some $t_0>0$ and some $\boldsymbol\theta\in \Delta_n$.
\end{enumerate}
Moreover, the following two statements are also equivalent.
\begin{enumerate}[(i)]
    \item  The L\'{e}vy measure satisfies $\nu((-\infty,0))>0$, or $\nu((x^{-1},\infty))$ is not concave on $(0,\infty)$;
    \item The infinitely divisible random variable with  triplet $(0, t_0\nu,\gamma_0)$ does not satisfy \eqref{Sub} for some $t_0>0$ and some $\boldsymbol\theta,\boldsymbol\eta\in \Delta_n$ satisfying $\boldsymbol\theta \preceq\boldsymbol\eta$. 
\end{enumerate}    
\end{corollary}

By Remark 25.14 of \cite{S99}, for $t>0$, $\p(X_t>x)\sim t\nu((x,\infty))$ as $x\to\infty$ if $\nu((-\infty,0))=0$, $\int_D |x|\nu(\d x)<\infty$, and $(\nu((1,\infty)))^{-1}\nu((1,x))$ is subexponential. For any subexponential random variable $Y$ satisfying \eqref{Sup} for all $\boldsymbol\theta\in\Delta_n$, there exists an infinitely divisible random variable $X$ with triplet $(0,\nu,0)$ such that $\nu((x,\infty))\sim \p(Y>x)$ as $x\to\infty$. Hence, $\p(X>x)\sim \p(Y>x)$ as $x\to\infty$. Since the L\'{e}vy measure can be modified on bounded sets without affecting this tail equivalence, in light of Corollary \ref{Cor:diversification}, we can choose $\nu$ such that \eqref{Sup} fails for $X$ and some $\boldsymbol\theta\in\Delta_n$. Thus, tail behavior alone does not determine the non-diversification phenomenon; the structure of the distribution, as encoded by its L\'{e}vy measure, also plays an essential role.

 We next check whether $|X|$ has an infinite mean under the conditions of Theorem \ref{th:main1}. To avoid trivial cases, we suppose $\nu((0,\infty))>0$.
\begin{proposition}\label{Prop:0}
Suppose $X$ is an infinitely divisible random variable  with generating triplet $(0, \nu, \gamma)$ satisfying  $\nu((-\infty,0))=0$ and $\nu((0,\infty))>0$. If $\nu((x^{-1},\infty))$ is subadditive or concave on $(0,\infty)$, we have $\mathbb E(|X|)=\infty$.   
\end{proposition}

 Proposition \ref{Prop:0} implies that any nontrivial infinitely divisible distribution satisfying the conditions in Theorem \ref{th:main1} necessarily has an infinite mean. This is consistent with the conclusions in Proposition 2 of \cite{CEW25} and Proposition 3.2 of \cite{ALO25}, both of which state that \eqref{Sup} only holds for $X$ with infinite mean if  $X$ is not a constant and $\max_{i=1}^n\theta_i<1$.


In practical applications, it is also important to identify conditions under which \eqref{Sup} and \eqref{Sub} hold strongly strictly, indicating that diversification leads to a strict deterioration in the sense of first-order stochastic dominance.  We next investigate this problem  when the L\'{e}vy measure has a density. 
\begin{proposition}\label{Prop:2} Let  $\{X_t\}_{t\geq 0}$ be a L\'{e}vy process with generating triplet $(0, \nu, \gamma_0)$ satisfying $\int_D|x|\nu(\d x)<\infty$ and $\nu((-\infty,0))=0$. Moreover, suppose $\nu(\d x)=x^{-1}g(x^{-1})\d x$ for a nonnegative $g$ on $(0,\infty)$.
\begin{enumerate}[(i)]
\item If $g$ is subadditive on $(0,\infty)$, then  \eqref{Sup} holds for all $X_t$ with $t>0$ and all  $\boldsymbol\theta\in \Delta_n$. Moreover, if $g$ is strictly subadditive on $(0,\infty)$, then \eqref{Sup} holds strongly strictly whenever $t>0$ and $\max_{i=1}^n\theta_i<1$.
\item If $g$ is concave on $(0,\infty)$, then \eqref{Sub} holds  for all $X_t$ with $t>0$ and all   $\boldsymbol\theta,\boldsymbol\eta\in \Delta_n$ satisfying $\boldsymbol\theta \preceq\boldsymbol\eta$. Moreover, if  $g$ is strictly concave on $(0,\infty)$, then \eqref{Sub} holds strongly strictly whenever $t>0$ and $\boldsymbol\theta \prec\boldsymbol\eta$.
\end{enumerate}
    \end{proposition}

It follows from the proof of Proposition \ref{Prop:2} that $\nu(\d x)=x^{-1}g(x^{-1})\d x$ with a nonnegative and subadditive $g$ on $(0,\infty)$ is a stronger condition than the subadditivity of $\nu((x^{-1},\infty))$ on $(0,\infty)$, while $\nu(\d x)=x^{-1}g(x^{-1})\d x$ with a nonnegative and concave $g$ on $(0,\infty)$ is a stronger condition than the concavity of $\nu((x^{-1},\infty))$ on $(0,\infty)$. However, those two stronger conditions are easier to understand and to be checked. More importantly, we can obtain the strong strictness of \eqref{Sup} and \eqref{Sub} under the strict subadditivity or concavity of $g$.

In Proposition \ref{Prop:2}, the L\'{e}vy measure has the form $\nu(\d x)=x^{-1}g(x^{-1})\d x$ on $(0,\infty)$. If $g$ is subadditive or concave, then \eqref{Sup} or \eqref{Sub} holds. As $g$ is nonnegative,  the  concavity of $g$ on $(0,\infty)$ ensures that $g$ is increasing on $(0,\infty)$, which further implies that $g(x^{-1})$ is decreasing on $(0,\infty)$.  This means that the infinitely divisible distributions considered in (ii) of Proposition \ref{Prop:2} belong to the class of self-decomposable distributions (also called L class), which is an extension of the stable distributions; see, e.g., Section 15 of \cite{S99} for the definitions, properties and characterization of the self-decomposable distributions. If $g(x)=xh(x),~x\in (0,\infty)$, where $h$ is decreasing on $(0,\infty)$, then $g$ is subadditive. Using this fact, we can construct many nonnegative and subadditive $g$ that is not increasing. For instance,  $g(x)=xe^{-x},~x\in(0,\infty)$. This suggests that the ID distributions considered in (i) of Proposition \ref{Prop:2} may not belong to the L class. This subtle difference also reflects that  the requirement in \eqref{Sub} is  materially stronger than \eqref{Sup}.

Let us now explore risk diversification for L\'{e}vy processes when $\int_{D}|x|\nu(\d x)=\infty$. We say a L\'{e}vy process $\{X_t\}_{t\geq 0}$ is symmetric if $X_t\overset{d}{=}-X_t$ for all $t>0$.  A   L\'{e}vy process with triplet $(0,\nu,\gamma)$ is symmetric if and only if  $\gamma=0$ and the L\'{e}vy measure $\nu$ is symmetric, i.e., $\nu(B)=\nu(-B)$ for all $B\in\mathcal B(\R)$.  The symmetric Cauchy distribution with parameter $c>0$ is given by
$$\mu(\d x)=\pi^{-1}c(x^2+c^2)^{-1}\d x,~x\in\R$$
with characteristic function $\hat{\mu}(z)=e^{-c|z|},~z\in\R.$
We obtain a characterization of symmetric 1-stable processes in the following result.
\begin{theorem}\label{Prop:S} Let  $\{X_t\}_{t\geq 0}$ be a symmetric L\'{e}vy process with generating triplet $(0, \nu, 0)$. Then we have the following conclusions.
\begin{enumerate}[(i)]
    \item  The stochastic dominance in \eqref{Sup} holds for all $X_t$ with $t>0$ and all $\boldsymbol\theta\in \Delta_n$ if and only if $\nu(\d x)=c|x|^{-2}\d x$ on $\R\setminus\{0\}$ for some $c\geq 0$.
    \item The stochastic dominance in \eqref{Sub} holds for all  $X_t$ with $t>0$  and  all  $\boldsymbol\theta,\boldsymbol\eta\in \Delta_n$ satisfying $\boldsymbol\theta \preceq\boldsymbol\eta$  if and only if $\nu(\d x)=c|x|^{-2}\d x$ on $\R\setminus\{0\}$ for some $c\geq 0$.
\end{enumerate}
\end{theorem}

Theorem \ref{Prop:S} reveals a strong rigidity within the class of symmetric L\'{e}vy processes: the symmetric 1-stable (Cauchy) process is the only nontrivial symmetric L\'{e}vy process exhibiting either phenomenon uniformly over all time horizons.  Based on Cauchy distributions,  \cite{M25} introduced the family of super-Cauchy distributions and showed that those distributions satisfy the stochastic dominance in \eqref{Sup}. The Cauchy distribution is a symmetric $\alpha$-stable distribution with $\alpha=1$. By direct verification, \cite{IB05} (see also \cite{M25} or \cite{CHSZ26}) showed that other 1-stable distributions with positive skewness parameters also satisfy the stochastic dominance in \eqref{Sup} and \eqref{Sub}. Moreover, \cite{IB09} studied the majorization order for Value-at-Risk with confidence levels above $1/2$ of the convex combination of iid symmetric random variables following the convolution of symmetric $\alpha$-stable distributions. The results in \cite{IB09} imply that the stochastic dominance in \eqref{Sup} and \eqref{Sub} do not hold for those special classes of symmetric infinitely divisible distributions, consistent with our Theorem \ref{Prop:S}.

 Applying Theorem \ref{th:main1} and Proposition \ref{Prop:2}, we can recover some results in the literature and also construct many new examples.
\begin{example}\label{Exam:1}
\begin{enumerate}[(i)]
    \item   The L\'{e}vy measure of a positive one-sided $\alpha$-stable distribution is given by  $\nu(\d x)=cx^{-1-\alpha}\d x$ on $(0,\infty)$ for some $c>0$ and $\alpha\in (0,1)$.  Note that the corresponding $g(x)=cx^{\alpha}$. It follows from Proposition \ref{Prop:2} that \eqref{Sup} and \eqref{Sub} hold strongly strictly whenever $\alpha\in (0,1)$. By Theorem \ref{Prop:S}, \eqref{Sup} and \eqref{Sub} hold for the symmetric $1$-stable distribution.  These recover the well-known result on stochastic dominance in \cite{IB05}.
\item Consider the infinitely divisible distributions with L\'{e}vy measure supported on $(0,\infty)$ given by $\nu(\d x)=\sum_{i=1}^nc_i x^{-1-\alpha_i} (\log(x^{-1}+1))^{\beta_i}\d x$ on $(0,\infty)$ for $\alpha_i\in (0,1)$ and $\beta_i\in [0,1]$ satisfying $\alpha_i+\beta_i\leq 1$. It follows from Proposition \ref{Prop:2} that \eqref{Sup} and \eqref{Sub} hold strongly strictly. If $\beta_i=0$, then it corresponds to the convolution of $n$ positive one-sided $\alpha$-stable distributions with heterogeneous $\alpha_i$.
\item More examples can be constructed using Proposition \ref{Prop:2} by setting $\nu(\d x)=x^{-1}g(x^{-1})\d x$ on $(0,\infty)$ and $\nu((-\infty,0))=0$ for a nonnegative concave function $g$ satisfying $\int_1^\infty x^{-2}g(x)\d x<\infty$. If $g(x)=x(1+x)^{-1}$, then $\nu(\d x)=(x(1+x))^{-1}\d x$ on $(0,\infty)$; if  $g(x)=1-e^{-x}$, then $\nu(\d x)=x^{-1}(1-e^{-x^{-1}})\d x$ on $(0,\infty)$. The corresponding infinitely divisible distributions satisfy \eqref{Sup} and \eqref{Sub} strongly strictly.
\end{enumerate}
    
\end{example}
Next, we consider a special case of L\'{e}vy processes, i.e. $\nu(\R)<\infty$. A compound Poisson process is defined by
\begin{align}\label{CP}
X_t=\sum_{i=1}^{N(t)} Y_i,~t\geq 0,
\end{align}
 where $Y_i$ are iid copies of $Y$ and are independent of the Poisson process $N(t)$ with parameter $\lambda>0$. Compound Poisson processes have been extensively studied in the literature and  applied to modeling the aggregate loss of an institution or a line of business over time in insurance and finance. Here $N(t)$ represents the number of losses over the time period $[0,t]$, while $Y_i$ represents the severity of the i-th loss. Compound Poisson processes are also among the most important examples of L\'{e}vy processes. The characteristic function is given by
$$\hat{\mu}_{X_t}(z)=\exp[t\lambda(\hat{\mu}_Y(z)-1)]=\exp\left[t\lambda\left(\int_\R (e^{izx}-1)\mu_Y(\d x)\right)\right],~z\in\R,t\geq 0.$$
Hence, the L\'{e}vy measure of $X_t$ is $t\lambda\mu_Y$, while the L\'{e}vy measure associated with the compound Poisson process is $\lambda\mu_Y$.

Applying Theorem \ref{th:main1}, we immediately arrive at the following stochastic dominance results for compound Poisson distributions.
\begin{proposition}\label{Prop:3}  Let $\{X_t\}_{t\geq 0}$ be a compound Poisson process defined by \eqref{CP}. Then the following conclusions hold.
\begin{enumerate}[(i)]
\item The stochastic dominance in \eqref{Sup} holds for all $X_t$ with $t>0$ and all $\boldsymbol\theta\in \Delta_n$ if and only if $1-F_Y(x^{-1})$ is subadditive on  $(0,\infty)$ and $F_Y(0-)=0$. Moreover, if $1-F_Y(x^{-1})$ is strictly subadditive on $(0,\infty)$, then \eqref{Sup} holds strictly  whenever $t>0$ and  $\max_{i=1}^n\theta_i<1$.
\item  The stochastic dominance in \eqref{Sub} holds for all $X_t$ with $t>0$ and  all  $\boldsymbol\theta,\boldsymbol\eta\in \Delta_n$ satisfying $\boldsymbol\theta \preceq\boldsymbol\eta$  if and only if $1-F_Y(x^{-1})$ is concave on  $(0,\infty)$ and $F_Y(0-)=0$. Moreover, if $1-F_Y(x^{-1})$ is nonlinear and concave on $(0,\infty)$, then  \eqref{Sub} holds strictly whenever $t>0$ and $\boldsymbol\theta \prec\boldsymbol\eta$.
\end{enumerate}
\end{proposition}

Note that the conclusions in  Proposition \ref{Prop:3} recover and extend Theorems 3-4 of \cite{CHSZ26}, where $F_Y(0)=0$ and the continuity of $F_Y$ are assumed.  An important difference is that we neither assume $Y$ is supported on $(0,\infty)$ nor the continuity of $F_Y$ over $(0,\infty)$. Instead,  $F_Y(0-)=0$ and continuity of $F_Y$ over $(0,\infty)$ appear as part of the conclusion in Proposition \ref{Prop:3}. Note that the continuity of $F_Y$ over $(0,\infty)$ is implied by the subadditivity of $1-F_{Y}(x^{-1})$ and the fact that $\lim_{x\downarrow 0}1-F_{Y}(x^{-1})=0$.   This means that if $Y$ takes negative values with  positive probability or $F_Y$ is not continuous over $\R\setminus\{0\}$, then the stochastic dominance in \eqref{Sup} and \eqref{Sub} do not hold for the corresponding compound Poisson random variable  $X_t$ for some $t>0$.
Moreover, Proposition \ref{Prop:3} provides sufficient conditions for the strict stochastic dominance in \eqref{Sup} and \eqref{Sub}.  The strong strictness of \eqref{Sup} and \eqref{Sub} for compound Poisson processes  can also be obtained from Proposition \ref{Prop:2}.
In addition, Proposition \ref{Prop:3} allows $F_Y(0)>0$, which is important in applications, as  it allows for a positive probability of zero loss. This is quite common in practice, for example, in insurance claims.  

It is worth mentioning that the condition for $Y$ in (i) of Proposition \ref{Prop:3}  coincides with the condition for $X$ in a recent paper \cite{ALO25} to study the stochastic dominance of \eqref{Sup}.  \cite{ALO25} show that if $1-F_X(x^{-1})$ is subadditive on $(0,\infty)$ with $F_X(0)=0$, then  \eqref{Sup} holds for all  $\boldsymbol\theta\in \Delta_n$. Those distributions are said to be \emph{inverted subadditive} (\emph{InvSub}).  Moreover, the distributions for $Y$ in (ii) of Proposition \ref{Prop:3} are said to be \emph{inverted concave} if $F_Y(0)=0$ in \cite{CHSZ26}, which is a stronger condition than the InvSub. Some Pareto, Fr\'{e}chet, and inverse-Gamma distributions are inverted concave; see \cite{CHSZ26} and \cite{ALO25} for more examples for these two classes of distributions.  


We next see some concrete examples satisfying the condition in Proposition \ref{Prop:3}. More examples can be found in e.g., \cite{ALO25} and \cite{CHSZ26}, and the references therein.
\begin{example}  Let $X_1,\dots,X_n$ be iid copies of $X$ defined by \eqref{CP} and suppose $Y$ is supported on $(0,\infty)$.
    \begin{enumerate}[(i)]
    \item If $Y$ has a super-Pareto distribution (see \cite{CEW25}), i.e., there exists an increasing, convex and non-constant function $f$ such that $Y\overset{d}{=}f(Z)$, where $Z\sim$ Pareto(1), then \eqref{Sup}  holds whenever $f(1)\geq 0$. If $Y\sim$ Pareto($\alpha$), then \eqref{Sup} holds whenever $\alpha\in (0,1]$.
    \item If $Y$ has a Fr\'{e}chet distribution $F(x)=e^{-x^{-\alpha}},x>0$ with $\alpha\in (0,1]$, then it is inverted concave. Hence both \eqref{Sup} and \eqref{Sub} hold. 
    \item If $Y$ has the inverse-Gamma distribution with density $f(x)=\beta^{\alpha}x^{-\alpha-1}e^{-\beta x^{-1}}/\Gamma(\alpha),~x>0$, $\alpha\in (0,1]$ and $\beta>0$, then it is inverted concave. Hence, both \eqref{Sup} and \eqref{Sub} hold.
    \end{enumerate}
\end{example}


\section{Multivariate infinitely divisible distributions}\label{Sec:multivariate}
In this section, we characterize the L\'{e}vy measures of  multidimensional L\'{e}vy processes exhibiting the adverse diversification phenomena. We first consider multidimensional L\'{e}vy processes with bounded variation sample paths and then investigate three important classes for which  more explicit characterizations can be obtained: multidimensional symmetric L\'{e}vy processes, multidimensional $\alpha$-stable processes and multidimensional compound Poisson processes.

Let $\{\mathbf X_t\}_{t\geq 0}$  be an $n$-dimensional L\'{e}vy process with generating triplet $(0,\nu,\boldsymbol\gamma_0)$
satisfying  $\int_{\R^n}\|\mathbf x\|^2\wedge 1\nu(\d \mathbf x)<\infty$. In this section, we assume $\boldsymbol \gamma_0=(\gamma_0,\dots,\gamma_0)$ such that all convex combinations have the same drift term.
For a multidimensional infinitely divisible random vector $\mathbf X=(X_1,\dots,X_n)$, we shall investigate 
\begin{align}\label{Supm}
  X_1\leq_{\mathrm{st}} \theta_1X_1+\dots+\theta_nX_n 
\end{align}
for $\boldsymbol\theta\in \Delta_n$
and
\begin{align}\label{Subm}
    \sum_{i=1}^n\eta_iX_i\leq_{\mathrm{st}} \sum_{i=1}^n\theta_iX_i
\end{align}
for  $\boldsymbol\theta,\boldsymbol\eta\in \Delta_n$ satisfying $\boldsymbol\theta \preceq\boldsymbol\eta$ .
Note that here $X_i,i\in[n]$ need not be independent or identically distributed. For $\boldsymbol\theta\in \Delta_n$,  denote by $T_{\boldsymbol\theta}\nu(B)=\nu(\{\mathbf x\in\R^n: \sum_{i=1}^n\theta_ix_i\in B\})$ for $B\in\mathcal B(\R\setminus\{0\})$. For $i\in [n]$, let  $\boldsymbol e_i$ be the $i$-th canonical basis vector of
$\mathbb R^n$.
\subsection{Multidimensional L\'{e}vy processes with bounded variation}
In this subsection, we consider the multidimensional L\'{e}vy processes with bounded variation sample paths, i.e., $\int_{\R^n}\|\mathbf x\|\wedge 1\nu(\d \mathbf x)<\infty$.
As a multidimensional counterpart of Theorem \ref{th:main1}, we next characterize the L\'{e}vy measures for the validity of \eqref{Supm} and \eqref{Subm} for  L\'{e}vy processes with bounded variation.
\begin{proposition}\label{th:main2} Let $\{\mathbf X_t\}_{t\geq 0}$ be an $n$-dimensional L\'{e}vy process with generating triplet $(0, \nu, \boldsymbol\gamma_0)$ satisfying $\int_{\R^n}\|\mathbf x\|\wedge 1\nu(\d \mathbf x)<\infty$.  Then the following conclusions hold.
\begin{enumerate}[(i)]
\item Fix $\boldsymbol\theta\in\Delta_n$. The stochastic dominance in \eqref{Supm} holds for all $\mathbf X_t$ with $t>0$  if and only if $\nu(\{\mathbf x\in\R^n: \sum_{i=1}^n\theta_ix_i>u\})\geq \nu(\{\mathbf x\in\R^n: x_1>u\})$ on $(0,\infty)$ and $\nu(\{\mathbf x\in\R^n: \sum_{i=1}^n\theta_ix_i<u\})\leq \nu(\{\mathbf x\in\R^n: x_1<u\})$ on $(-\infty,0)$. Moreover, if $T_{\boldsymbol\theta}\nu\neq T_{e_1}\nu$, then \eqref{Supm} holds strictly whenever $t>0.$
\item Fix  $\boldsymbol\theta,\boldsymbol\eta\in \Delta_n$ satisfying $\boldsymbol\theta \preceq\boldsymbol\eta$ . The stochastic dominance in \eqref{Subm} holds for all $\mathbf X_t$ with $t>0$ if and only if  $\nu(\{\mathbf x\in\R^n: \sum_{i=1}^n\theta_ix_i>u\})\geq \nu(\{\mathbf x\in\R^n: \sum_{i=1}^n\eta_ix_i>u\})$ on $(0,\infty)$ and $\nu(\{\mathbf x\in\R^n: \sum_{i=1}^n\theta_ix_i<u\})\leq \nu(\{\mathbf x\in\R^n: \sum_{i=1}^n\eta_ix_i<u\})$ on $(-\infty,0)$. Moreover, if $T_{\boldsymbol\theta}\nu\neq T_{\boldsymbol\eta}\nu$, then \eqref{Subm} holds strictly whenever $t>0.$
\end{enumerate}
\end{proposition}

Unlike Theorem \ref{th:main1},  the conditions on the L\'{e}vy measures given in Proposition \ref{th:main2} look complicated due to the complexity  and generality of the multivariate L\'{e}vy measures. However, Proposition \ref{th:main2} is indeed very useful as a starting point to obtain more explicit characterizations of L\'{e}vy measures when we study some specific cases. Note that the constraint  $\int_{\R^n}\|\mathbf x\|\wedge 1\nu(\d \mathbf x)<\infty$ is not needed for the ``only if'' parts of Proposition \ref{th:main2}. This is important when we investigate general L\'{e}vy processes. 
\subsection{Symmetric L\'{e}vy processes}
In this subsection, we consider symmetric multidimensional L\'{e}vy processes, i.e., $\mathbf X_t\overset{d}{=}-\mathbf X_t$ for $t>0$,  without the constraint  $\int_{\R^n}\|\mathbf x\|\wedge 1\nu(\d \mathbf x)<\infty$. We first arrive at the following characterization result.

\begin{theorem}\label{prop:symmetric} Let $\{\mathbf X_t\}_{t\geq 0}$ be an $n$-dimensional symmetric L\'{e}vy process with generating triplet $(0, \nu, \boldsymbol0)$.  Then the following statements are equivalent.
\begin{enumerate}[(i)]
\item  The stochastic dominance in \eqref{Supm} holds for all $\mathbf X_t$ with $t>0$ and all $\boldsymbol\theta\in\Delta_n$; 
\item The stochastic dominance in \eqref{Subm} holds for all $\mathbf X_t$ with $t>0$ and all   $\boldsymbol\theta,\boldsymbol\eta\in \Delta_n$ satisfying $\boldsymbol\theta \preceq\boldsymbol\eta$;
\item The L\'{e}vy measure satisfies $T_{\boldsymbol\theta}\nu=T_{\boldsymbol\eta}\nu$ for all $\boldsymbol\theta, \boldsymbol\eta\in \Delta_n$.
\end{enumerate}
\end{theorem}

Theorem \ref{prop:symmetric} indicates that the symmetric  L\'{e}vy processes are very special from the perspective of  risk diversification. First, for symmetric L\'{e}vy processes, the validity of \eqref{Supm} and \eqref{Subm} is equivalent. Moreover,  the validity of \eqref{Supm} and \eqref{Subm} requires that $\sum_{i=1}^n\theta_iX_t^{(i)}\overset{d}{=}\sum_{i=1}^n\eta_iX_t^{(i)}$ for all $\boldsymbol\theta, \boldsymbol\eta\in \Delta_n$. Thus, within the symmetric class, uniform adverse diversification is extremely restrictive: the stochastic dominance relations in \eqref{Supm} and \eqref{Subm} necessarily reduce to equality in distribution, meaning that risk diversification has no effect on the distribution of the aggregate risk.  Here we do not require $\int_{\R^n}\|\mathbf x\|\wedge 1\nu(\d \mathbf x)<\infty$. Hence, the conclusion is valid for all symmetric L\'{e}vy processes. Moreover, if $X_t^{(i)},~i\in[n]$ are iid, then it reduces to the setting of Theorem \ref{Prop:S}; in the non-degenerate case,  $X_t^{(i)}$ follows a symmetric Cauchy distribution. The property $T_{\boldsymbol\theta}\nu=T_{\boldsymbol\eta}\nu$ for all $\boldsymbol\theta, \boldsymbol\eta\in \Delta_n$ can be referred to as  \emph{projection invariance}. Moreover, the L\'{e}vy measure in Theorem \ref{prop:symmetric} is also symmetric, as implied by the symmetry of $\mathbf X_t$ for $t>0$.  Hence, the relevant L\'{e}vy measures are symmetric and  projection-invariant. In what follows, we will apply Theorem \ref{prop:symmetric} to multidimensional symmetric 1-stable processes.

Next, we consider  multivariate $\alpha$-stable distributions. For a multivariate $\alpha$-stable distribution with $\alpha\in (0,2)$, the L\'{e}vy measure is given by \begin{align}\label{stablev}\nu(B)=\int_{S_n}\int_0^\infty r^{-1-\alpha}\id_B(r\boldsymbol \xi)\d r\rho(\d \boldsymbol \xi)
\end{align}
for $B\in\mathcal{B}(\R^n)$, where $S_n$ represents the unit sphere in $\R^n$ and $\rho$ is a finite measure on $S_n$ called the spectral measure; see e.g., Theorem 14.3 of \cite{S99}. Note that $\nu$ and $\rho$ are mutually determined.  We denote by $S_n^+$ the portion of $S_n$ lying in the positive orthant and $S_n^-$ the portion of $S_n$ lying in the negative orthant.

Applying Theorem \ref{prop:symmetric}, we obtain the following result for  symmetric 1-stable processes, and the general $\alpha$-stable processes will be considered in the next subsection.
\begin{proposition}\label{Prop:SM1} Let $\{\mathbf X_t\}_{t\geq 0}$ be an $n$-dimensional $1$-stable process with generating triplet $(0, \nu, \mathbf 0)$ with $\nu$ given by \eqref{stablev}. Suppose the process is symmetric.  Then the following statements are equivalent.
\begin{enumerate}[(i)]
\item  The stochastic dominance in \eqref{Supm}  holds for all $\mathbf X_t$ with $t>0$ and all $\boldsymbol\theta\in\Delta_n$; 
\item The stochastic dominance in \eqref{Subm} holds for all $\mathbf X_t$ with $t>0$ and all  $\boldsymbol\theta,\boldsymbol\eta\in \Delta_n$ with $\boldsymbol\theta \preceq\boldsymbol\eta$;
\item The spectral measure $\rho$ is symmetric and supported on $\R_+^n\cup\R_-^n$, and  $\int_{S_n}(\xi_i)_+\rho(\d \boldsymbol\xi), \int_{S_n}(\xi_i)_-\rho(\d \boldsymbol\xi),~i\in[n]$ are all equal.
\end{enumerate}
\end{proposition}

Proposition \ref{Prop:SM1} shows that  multidimensional symmetric 1-stable processes  exhibit the non-diversification phenomenon or the reverse diversification order if and only if the spectral measure is supported on the positive and negative orthants and the integrals of positive parts and negative parts under the spectral measure are identical across all components regardless of the dependence among those components.
\subsection{Multidimensional $\alpha$-stable processes}
In this subsection, we focus on the multidimensional $\alpha$-stable processes.
 We say a measure $\rho$ on $S_n$ is \emph{permutation-invariant} if $\rho(T_{\sigma}(B))=\rho(B)$ for all permutation $\sigma$, where $B\in \mathcal B(S_n)$ and $T_{\sigma}(\mathbf x)=(x_{\sigma(1)},\dots,x_{\sigma(n)})$ with $\sigma$ ranging over all permutations of $(1,\dots,n)$.
Inspired by Theorem \ref{th:main1}, we suppose $\nu$ is supported on the positive orthant in the following result.

\begin{theorem}\label{Thm:4} Let $\{\mathbf X_t\}_{t\geq 0}$ be an $n$-dimensional $\alpha$-stable process with  $\alpha\in (0,1)$ and generating triplet $(0, \nu, \boldsymbol\gamma_0)$ with $\nu$ given by \eqref{stablev} satisfying $\nu(\R^n\setminus\R_+^n)=0$. Then the following conclusions hold.
\begin{enumerate}[(i)]
\item  The stochastic dominance in \eqref{Supm} holds for all $\mathbf X_t$ with $t>0$ and $\boldsymbol\theta\in\Delta_n$ if and only if $\min_{i\in [n]}\int_{S_n^+}\xi_i^\alpha\rho(\d \boldsymbol\xi)$ $=\int_{S_n^+}\xi_1^\alpha\rho(\d \boldsymbol\xi)$. 
\item The stochastic dominance in \eqref{Subm} holds for all $\mathbf X_t$ with $t>0$ and    $\boldsymbol\theta,\boldsymbol\eta\in \Delta_n$ with $\boldsymbol\theta \preceq\boldsymbol\eta$ if and only if $\rho$ is permutation-invariant. 
\end{enumerate}
\end{theorem}

Theorem \ref{Thm:4} provides sufficient and necessary conditions on the L\'{e}vy measures of the $\alpha$-stable processes for the validity of \eqref{Supm} and \eqref{Subm}, where $\nu$ is supported on the positive orthant of $\R^n$. It was shown in \cite{IB05} (see also \cite{M25} or \cite{CHSZ26}) that for iid $\alpha$-stable random variables with $\alpha\in (0,1)$, \eqref{Supm} and \eqref{Subm} are only valid if the corresponding L\'{e}vy measure is supported on the positive orthant. A similar support requirement also arises in the general bounded-variation iid setting  in Theorem \ref{th:main1}. Moreover, in many applications in risk management, losses are nonnegative. Hence, it is reasonable to restrict the L\'{e}vy measure to be supported on the positive orthant.  

When  $\nu(\R^n\setminus\R_+^n)>0$, a similarly simple characterization to those of Theorem \ref{Thm:4} is difficult to obtain. In this general case, the reverse diversification order can be characterized by two opposite Schur properties associated with the positive and negative parts of the spectral measure.  The precise characterization is provided in Proposition \ref{Prop:additinal} in Section \ref{Appendix:additional} of the Appendix.

Moreover, the conditions on the L\'{e}vy measures show the complexity of the multivariate $\alpha$-stable distributions as they involve the dependence structure among the components. The conclusion in (ii) of Theorem \ref{Thm:4} shows that the validity of \eqref{Subm} requires the permutation-invariance of $\rho$, which is equivalent to the permutation-invariance of $\nu$ and  the exchangeability of $\mathbf X_t$. Hence, Theorem \ref{Thm:4} means that the multidimensional $\alpha$-stable processes with $\alpha\in (0,1)$ satisfy \eqref{Subm} if and only if the process is exchangeable, reflecting the identical marginals and symmetric dependence structure.  The conclusion in  (i) of Theorem \ref{Thm:4} shows that the validity of \eqref{Supm}  depends only on the relationship among the  $\alpha$-th moments of each component under $\rho$, i.e., the first component has the smallest $\alpha$-th moment.  This is equivalent to  $$\min_{i\in[n]}\int_{\R_+^n\cap\{\|\mathbf x\|>1\}}x_i^\alpha\nu(\d \mathbf x)=\int_{\R_+^n\cap\{\|\mathbf x\|>1\}}x_1^\alpha\nu(\d \mathbf x).$$
This means that if the first component has the smallest $\alpha$-th moment with respect to the L\'{e}vy measure on large positive jumps, then \eqref{Supm} holds.

Next, we see some simple examples to illustrate the dependence structures of the $\alpha$-stable distributions considered in  Theorem~\ref{Thm:4}.

\begin{example} 
\begin{enumerate}[(i)]
\item Suppose that the spectral measure is concentrated on the
coordinate axes:
$\rho_{\mathrm{axis}}=c\sum_{i=1}^n\delta_{\boldsymbol e_i}$ with $c>0$. Then the components of $\boldsymbol X_t$ are
independent and identically distributed positive $\alpha$-stable
random variables.
Clearly, $\rho_{\mathrm{axis}}$ is permutation-invariant. Hence, in light of Theorem \ref{Thm:4}, the stochastic dominance relations in \eqref{Supm} and
\eqref{Subm} hold, recovering the result in \cite{IB05}.
\item Let
$ 
\xi_0=\frac{1}{\sqrt n}(1,\dots,1)
$ 
and suppose that the spectral measure is concentrated on the
diagonal direction:
$ 
\rho_{\mathrm{diag}}=c\delta_{\xi_0}
$ for some $c>0$.
In this case,
$X_{t}^{(1)}=\cdots=X_{t}^{(n)}$ a.s.
Hence, the components are completely dependent. Clearly, $\rho_{\mathrm{diag}}$ is permutation-invariant. Both \eqref{Supm} and \eqref{Subm} hold as equalities.
\item Let $\rho_\lambda=\lambda \rho_{\mathrm{axis}}+(1-\lambda)\rho_{\mathrm{diag}}$ for $\lambda\in (0,1)$.  The components of the corresponding $\alpha$-stable process exhibit dependence intermediate between complete dependence and independence.  Note that $\rho_\lambda$ is permutation-invariant. Hence,  both \eqref{Supm} and \eqref{Subm} hold. 
\end{enumerate}
\end{example}

Next, we explore  the stochastic dominance in \eqref{Supm} and \eqref{Subm} for L\'{e}vy processes with independent and heterogeneous $\alpha$-stable components.
\begin{proposition}\label{prop:heter} Suppose $\{X_t^{(i)}\}_{t\geq 0}, ~i\in [n]$ are independent $\alpha_i$-stable  processes with $\alpha_i\in (0,1)$ and generating triplets $(0,\nu_i,\gamma_0)$ satisfying $\nu_i((-\infty,0))=0,~i\in[n]$ and $\nu_1\neq 0$.
Then the following conclusions hold.
\begin{enumerate}[(i)]
\item  The stochastic dominance in \eqref{Supm}  holds for all $(X_t^{(1)},\dots,X_t^{(n)})$ with $t>0$ and all $\boldsymbol\theta\in\Delta_n$ if and only if $\alpha_1=\dots=\alpha_n$ and $\nu_i((1,\infty))\geq \nu_1((1,\infty)),~i\in [n]$.
\item The stochastic dominance in \eqref{Subm} holds for all $(X_t^{(1)},\dots,X_t^{(n)})$ with $t>0$ and all  $\boldsymbol\theta,\boldsymbol\eta\in \Delta_n$ with $\boldsymbol\theta \preceq\boldsymbol\eta$ if and only if $\alpha_1=\dots=\alpha_n$ and $\nu_i((1,\infty))=\nu_1((1,\infty)),~i\in [n]$.
\end{enumerate}
\end{proposition}

Proposition \ref{prop:heter} implies that, for a L\'{e}vy process with independent heterogeneous $\alpha$-stable components,  \eqref{Supm} or \eqref{Subm} can hold only if all components have the same stability index. 

\subsection{Multidimensional compound Poisson processes}
In this subsection, we consider the case $\nu(\R^n)<\infty$. An $n$-dimensional compound Poisson process is defined by
\begin{align}\label{CPM}
\mathbf X_t=\sum_{i=1}^{N(t)} \mathbf Y_i,~t\geq 0,
\end{align}
 where $\mathbf Y_i$ are iid copies of $\mathbf Y$ and are independent of the Poisson process $N(t)$ with parameter $\lambda>0$. Here $\mathbf Y=(Y_1,\dots,Y_n)$ is a random vector.  The characteristic function is given by
$$\hat{\mu}_{\mathbf X_t}(\mathbf z)=\exp[t\lambda(\hat{\mu}_{\mathbf Y}(\mathbf z)-1)]=\exp\left[t\lambda\left(\int_{\R^n} (e^{i\mathbf z\mathbf x}-1)\mu_{\mathbf Y}(\d \mathbf x)\right)\right],~\mathbf z\in\R^n,t\geq 0$$
with the L\'{e}vy measure  $\lambda\mu_{\mathbf Y}$.

Applying Proposition \ref{th:main2}, we immediately arrive at the following stochastic dominance results for multivariate compound Poisson distributions.
\begin{proposition}\label{Prop:l}  Let $\{\mathbf X_t\}_{t\geq 0}$ be an $n$-dimensional compound Poisson process defined by \eqref{CPM}. Then  the following conclusions hold.
\begin{enumerate}[(i)]
\item The stochastic dominance in \eqref{Supm} holds for all $\mathbf X_t$ with $t>0$  and all $\boldsymbol\theta\in \Delta_n$  if and only if \eqref{Supm} holds for $\mathbf Y$ with all $\boldsymbol\theta\in \Delta_n$.
\item The stochastic dominance in \eqref{Subm} holds for all $\mathbf X_t$ with $t>0$ and all $\boldsymbol\theta,\boldsymbol\eta\in \Delta_n$ satisfying $\boldsymbol\theta \preceq\boldsymbol\eta$  if and only if \eqref{Subm} holds for $\mathbf Y$ with all $\boldsymbol\theta,\boldsymbol\eta\in \Delta_n$ satisfying $\boldsymbol\theta \preceq\boldsymbol\eta$.
\end{enumerate}
\end{proposition}

Note that the components of a multidimensional compound Poisson process share the same counting process for each component. Therefore, Proposition \ref{Prop:l} is fundamentally different from Proposition \ref{Prop:3}, even when the components of $\mathbf Y$ are assumed to be iid.   Proposition \ref{Prop:l} reveals a fundamental correspondence between the multidimensional compound Poisson process and its loss random vector regarding the validity of \eqref{Supm} and \eqref{Subm}.  Specifically, the random loss vector satisfies \eqref{Supm} or \eqref{Subm} if and only if the corresponding compound Poisson process satisfies the respective stochastic order for all $t>0$. This correspondence may also provide a useful way to construct further examples satisfying \eqref{Supm} and \eqref{Subm}. 

Next, we consider $\mathbf X_t=(X_t^{(1)},\dots, X_t^{(n)})$, where $\{X_t^{(i)}\}_{t\geq 0}$ are independent compound Poisson processes with heterogeneous distributions given by
\begin{align}\label{ICP}
    X_t^{(i)}=\sum_{j=1}^{N_i(t)} Y_{ij},~i\in [n],
\end{align}
where $N_i(t)$ are the Poisson processes with parameters $\lambda_i>0$, $\{Y_{ij}\}_{j=1}^\infty$ representing  the claims of the $i$-th line of business are iid copies of $Y_i$, and all of these processes and random variables  are independent. 

\begin{proposition}\label{Prop:ICP}  Let $\{\mathbf X_t\}_{t\geq 0}$ be a L\'{e}vy process with independent compound Poisson components defined by \eqref{ICP}. Then the following conclusions hold.
\begin{enumerate}[(i)]
\item  The stochastic dominance in \eqref{Supm} holds for all $\mathbf X_t$ with $t>0$ and all $\boldsymbol\theta\in \Delta_n$   if and only if  $F_{Y_2}(0-)=\cdots=F_{Y_n}(0-)=0$ and $\lambda_1(1-F_{Y_1}((x_1+\dots+x_n)^{-1}))\leq \sum_{i=1}^n\lambda_i(1-F_{Y_i}(x_i^{-1}))$ for all $x_i\in (0,\infty)$.
\item The stochastic dominance in \eqref{Subm} holds for all $\mathbf X_t$ with $t>0$ and all $\boldsymbol\theta,\boldsymbol\eta\in \Delta_n$ satisfying $\boldsymbol\theta \preceq\boldsymbol\eta$  if and only if $\lambda_1=\dots=\lambda_n$, $F_{Y_1}=\dots=F_{Y_n}$, $F_{Y_1}(0-)=0$  and $1-F_{Y_1}(x^{-1})$ is concave on  $(0,\infty)$.
\end{enumerate}
\end{proposition}

Note that Proposition \ref{Prop:ICP} is an extension of the results in Proposition \ref{Prop:3}. For heterogeneous compound Poisson processes, the validity of \eqref{Supm} does not require $F_{Y_1}(0-)=0$, thereby allowing negative values for $Y_1$, in contrast to Proposition \ref{Prop:3}.  Moreover, a more general form of subadditivity is required. In contrast, the validity of \eqref{Subm} forces the heterogeneous compound Poisson processes to reduce to the homogeneous case.

\section{Extensions and applications}\label{Sec:extension}
The purpose of this section is not merely to construct further examples satisfying the stochastic dominance relations, but to show that uniform-in-time L\'{e}vy-process results extend to sample-path-dependent functionals relevant to actuarial science, finance and operations research.

For  a L\'{e}vy process $\{X_t\}_{t\geq 0}$, we consider some stochastic processes dependent on the  sample paths of the  L\'{e}vy process.
For $t\geq 0$, define the following three new stochastic processes $$Y_t=\sup_{0\leq s\leq t}X_s, \quad \quad Z_t=\int_0^t f(s,t)\d X_s\quad \text{and}\quad  W_t=\int_0^t g(X_s)\d s ,$$
where $f$ and $g$ are functions satisfying the conditions specified below. 
Next, we investigate whether the above  stochastic processes inherit the stochastic dominance in \eqref{Sup} or \eqref{Sub} from the original L\'{e}vy process, and discuss their applications and implications.
\subsection{Running maxima}
In this subsection, we consider the running maximum of a L\'{e}vy process defined by $$Y_t=\sup_{0\leq s\leq t}X_s,~t\geq 0.$$  The running maxima of L\'{e}vy processes  are fundamental quantities in fluctuation theory and arise naturally in insurance risk, storage, and first-passage problems; see e.g.,  Chapter XI of \cite{AA10} for applications to insurance risk, \cite{DM15} for L\'{e}vy-driven storage models, and Chapter 7 of \cite{K14} for  first-passage problems. We investigate the diversification properties of the running maximum and discuss their implications. We first establish the following result.

\begin{proposition}\label{Prop:ruin}
    Let  $\{X_t\}_{t\geq 0}$ be a L\'{e}vy process with generating triplet $(0, \nu, \gamma_0)$ satisfying $\int_D|x|\nu(\d x)<\infty$ and $\nu((-\infty,0))=0$, and  $\{X_t^{(i)}\}_{t\geq 0},~i\in [n]$ be iid copies of $\{X_t\}_{t\geq 0}$. If  $\nu((x^{-1},\infty))$ is subadditive on $(0,\infty)$, then 
    \begin{align*}
Y_t\leq_{\mathrm{st}} \sup_{0\leq s\leq t}\sum_{i=1}^n\theta_iX_s^{(i)}\leq_{\mathrm{a.s.}} \sum_{i=1}^n\theta_i Y_t^{(i)}
    \end{align*}
    for all $t>0$ and all $\boldsymbol\theta\in \Delta_n$, where $Y_t^{(i)}$ represents the running maximum of $\{X_t^{(i)}\}_{t\geq 0}$.
\end{proposition}

Proposition \ref{Prop:ruin}  shows that the non-diversification phenomenon extends to running maxima. More precisely, the running maximum of the original process is stochastically smaller than the running maximum of the diversified process, which is further bounded above almost surely by the corresponding convex combination of the individual running maxima. In risk theory, $Y_t$ represents the largest cumulative loss over the time period $[0,t]$. Thus, the result implies that allocating a fixed exposure across independent copies of the  underlying L\'{e}vy loss process does not reduce  the resulting maximum-loss risk in the sense of first-order stochastic dominance whenever the underlying L\'{e}vy process exhibits a non-diversification phenomenon. 

Note that if $\gamma_0\geq 0$, then $Y_t=X_t$ a.s. The more interesting case is $\gamma_0<0$, in which the running maximum process $\{Y_t\}_{t\geq 0}$ is no longer a L\'{e}vy process in general since it is a sample-path-dependent transformation of $\{X_t\}_{t\geq 0}$. Therefore, Proposition \ref{Prop:ruin} extends the non-diversification phenomenon beyond the class of L\'{e}vy processes and, in particular,  provides new distributions satisfying \eqref{Sup}. 

The first inequality in Proposition \ref{Prop:ruin} shows that the running maximum of a diversified portfolio of L\'{e}vy risk processes stochastically dominates the running maximum of the underlying L\'{e}vy  process whenever the latter exhibits a non-diversification phenomenon. This has an important implication in risk theory. The surplus process of an insurer is commonly defined as 
$$u+ct-X_t,~t>0,$$
where $u>0$ is the initial capital, $c>0$ is the premium rate, and $X_t$ represents the cumulative loss up to time $t$. The cumulative loss process $X_t$ is often modeled by  a compound Poisson process or, more generally, a L\'{e}vy process. The  ruin probability over a finite time horizon  $[0,T]$ with $T>0$ is defined as 
$$\mathbb P\left(\inf_{0\leq t\leq T}(u+ct-X_t)<0\right)=\mathbb P\left(\sup_{0\leq t\leq T}(X_t-ct)>u\right).$$
Since $\{X_t-ct\}_{t\geq 0}$ is again a L\'{e}vy process with the same L\'{e}vy measure,
 Proposition \ref{Prop:ruin} yields the following result.
\begin{corollary}\label{Cor:ruin} Under the assumption of Proposition \ref{Prop:ruin} for the L\'{e}vy process $\{X_t\}_{t\geq 0}$ and its iid copies $\{X_t^{(i)}\}_{t\geq 0},~i\in [n]$, we have  
    $$\mathbb P\left(\inf_{0\leq t\leq T}(u+ct-X_t)<0\right)\leq \mathbb P\left(\inf_{0\leq t\leq T}\left(u+ct-\sum_{i=1}^n\theta_iX_t^{(i)}\right)<0\right)$$
    holds for all $\boldsymbol\theta\in\Delta_n$ and all $T>0$, where $u>0$ and $c>0$.
\end{corollary}

Corollary \ref{Cor:ruin} shows that diversifying a fixed exposure  across independent and identically distributed insurance risks may increase the  finite-horizon  ruin probability if the underlying L\'{e}vy loss process exhibits a non-diversification phenomenon.

For a L\'{e}vy process $\{X_t\}_{t\geq 0}$, the initially empty storage process  is given by 
$$Q_t=\sup_{0\leq s\leq t} \{X_t-X_s-c(t-s)\},$$
where $X_t$ represents the cumulative input and $c>0$ is the service rate. We refer to \cite{DM15}  for further details on the definition, formulation, and interpretation of storage processes.   By the stationary and independent increments of a L\'{e}vy process,  $$\{X_t-X_{t-s}-cs\}_{0\leq s\leq t}\overset{d}{=}\{X_s-cs\}_{0\leq s\leq t},~t\geq 0.$$ Hence, Proposition \ref{Prop:ruin} can be applied to $Q_t$, yielding the following result.
\begin{corollary}\label{Cor:queue}
  Let  $\{X_t\}_{t\geq 0}$ be a L\'{e}vy process with generating triplet $(0, \nu, \gamma_0)$ satisfying the assumptions of Proposition \ref{Prop:ruin}, and let  $\{X_t^{(i)}\}_{t\geq 0},~i\in [n]$ be iid copies of $\{X_t\}_{t\geq 0}$. If  $\nu((x^{-1},\infty))$ is subadditive on $(0,\infty)$, then, 
  $$Q_t\leq_{\mathrm{st}} \sup_{0\leq s\leq t} \left\{\sum_{i=1}^n\theta_i(X^{(i)}_t-X^{(i)}_s)-c(t-s)\right\}\leq_{\mathrm{a.s.}} \sum_{i=1}^n\theta_i Q_t^{(i)}$$
  for all $t>0$ and all $\boldsymbol\theta\in \Delta_n$, where $Q_t^{(i)}$ denotes the storage content associated with $\{X_t^{(i)}\}_{t\geq 0}$ and service rate $c$. 
\end{corollary}
 Corollary \ref{Cor:queue} shows that diversifying a fixed total input exposure across independent and identically distributed L\'{e}vy input streams may increase the storage content when the underlying L\'{e}vy input process exhibits a non-diversification phenomenon. Corollary \ref{Cor:queue} also  implies that such diversification may  increase the overload probability.

\subsection{L\'{e}vy-driven stochastic integral}
In this subsection, we consider the L\'{e}vy-driven stochastic integral defined by $$Z_t=\int_0^t f(s,t)\d X_s,~t\geq 0,$$ where $f(\cdot, t)$ is a continuous function and $\int_D|x|\nu(\d x)<\infty$. Under the assumption of $\int_D|x|\nu(\d x)<\infty$, the sample paths of a L\'{e}vy process have  bounded variation on every compact time interval almost surely. Hence, $Z_t$ is well-defined and finite almost surely for every $t\geq 0$. 

We investigate whether the non-diversification phenomenon and the reverse diversification order of the underlying L\'{e}vy process can be preserved under such stochastic integrals and discuss the corresponding applications. We first establish the following result.
\begin{proposition}\label{Prop:SI} Let  $\{X_t\}_{t\geq 0}$ be a L\'{e}vy process with generating triplet $(0, \nu, \gamma_0)$ satisfying $\int_D|x|\nu(\d x)<\infty$ and $\nu((-\infty,0))=0$.  
 Suppose $f(\cdot,t)$ is nonnegative and continuous on $[0,t]$ for every $t>0$. 
 \begin{enumerate}[(i)]
 \item If $\nu((x^{-1},\infty))$ is subadditive on $(0,\infty)$, then \eqref{Sup} holds for all $Z_t$ with $t>0$ and all $\boldsymbol \theta\in\Delta_n$;
 \item If $\nu((x^{-1},\infty))$ is concave on $(0,\infty)$, then  \eqref{Sub} holds for all $Z_t$ with $t>0$ and all $\boldsymbol \theta,\boldsymbol\eta\in\Delta_n$ satisfying $\boldsymbol\theta \preceq\boldsymbol\eta$.
 \end{enumerate}
\end{proposition}
Proposition \ref{Prop:SI} shows that the stochastic dominance properties in \eqref{Sup} and \eqref{Sub} can be preserved under L\'{e}vy-driven stochastic integration with nonnegative deterministic kernels. 
Clearly, $\{Z_t\}_{t\geq 0}$ is no longer a L\'{e}vy process in general since it is a sample-path-dependent transformation of $\{X_t\}_{t\geq 0}$.  Hence, Proposition \ref{Prop:SI}  extends the non-diversification phenomenon and the reverse diversification order established in Theorem \ref{th:main1} beyond the class of L\'{e}vy processes.

 The L\'{e}vy-driven Ornstein-Uhlenbeck (OU) process is defined by
$$\d V_t=-c V_t\d t+\d X_t,$$
where $c>0$ and $X_t$ is a L\'{e}vy process. The L\'{e}vy-driven OU processes constitute an important class of mean-reverting models and have found applications in areas such as financial econometrics, stochastic volatility, and commodity-price modeling; see e.g., \cite{BS01} and \cite{H09}. 
Suppose $V_0=x_0$ for some deterministic $x_0\geq 0$. Then the OU process has the representation
$$V_t=e^{-ct}x_0+\int_0^{t}e^{-c(t-s)}\d X_s,~t\geq 0.$$
Moreover, for $t>0$, the integrated OU process is given by
$$\int_0^t V_s\d s=x_0c^{-1}(1-e^{-ct})+\int_0^tc^{-1}(1-e^{-c(t-s)})\d X_s.$$
Applying Proposition \ref{Prop:SI}, we immediately obtain the following result.
\begin{corollary}\label{Cor:OU} Let  $\{X_t\}_{t\geq 0}$ be a L\'{e}vy process with generating triplet $(0, \nu, \gamma_0)$ satisfying $\int_D|x|\nu(\d x)<\infty$ and $\nu((-\infty,0))=0$.  
 \begin{enumerate}[(i)]
 \item If $\nu((x^{-1},\infty))$ is subadditive on $(0,\infty)$, then \eqref{Sup} holds for  $V_t$ and $\int_0^tV_s\d s$ for all  $t>0$ and all $\boldsymbol \theta\in\Delta_n$;
 \item If $\nu((x^{-1},\infty))$ is concave on $(0,\infty)$, then  \eqref{Sub} holds for $V_t$ and $\int_0^tV_s\d s$ for all $t>0$ and all $\boldsymbol \theta,\boldsymbol\eta\in\Delta_n$ satisfying $\boldsymbol\theta \preceq\boldsymbol\eta$.
 \end{enumerate}
\end{corollary}

Corollary \ref{Cor:OU} shows that the non-diversification phenomenon and the reverse diversification order of the underlying L\'{e}vy process can be  preserved under exponential mean reversion and temporal integration. In particular, the mean-reverting mechanism of the OU process does not necessarily restore diversification benefits.

Interestingly,  \cite{BS01} consider convex combinations of independent OU processes to construct more general volatility models.  When the L\'{e}vy driver is a subordinator so that the OU process remains nonnegative, $V_t$ can be interpreted as a variance factor. In this case, Corollary \ref{Cor:OU} shows that combining independent and identically distributed variance components may lead to a stochastically larger instantaneous variance and  integrated variance under the heavy-tailed L\'{e}vy drivers considered here.
\subsection{Integrated convex functional}
In this subsection, we consider the integrated convex functional of the L\'{e}vy process defined by $$W_t=\int_0^t g(X_s)\d s,~t\geq 0,$$ where $g$ is an increasing and convex function. As the sample paths of a L\'{e}vy process are right-continuous with left limits, they are almost surely bounded on every compact time interval. Since a finite convex function on $\R$ is continuous, $W_t$ is well-defined and finite almost surely for every $t\geq 0$.

We have the following result.
\begin{proposition}\label{Prop:integrated} Let  $\{X_t\}_{t\geq 0}$ be a L\'{e}vy process with generating triplet $(0, \nu, \gamma_0)$ satisfying $\int_D|x|\nu(\d x)<\infty$ and $\nu((-\infty,0))=0$.  
If $g$ is increasing and convex on $\R$, and  $\nu((x^{-1},\infty))$ is subadditive on $(0,\infty)$, then \eqref{Sup} holds for all $W_t$ with $t>0$ and  all $\boldsymbol \theta\in\Delta_n$.
\end{proposition}

An important special case is $g(x)=(x-K)_+$, for which, $$W_t=\int_0^t (X_s-K)_+\d s,~t\geq 0.$$ This quantity can be interpreted as the cumulative excess loss above the threshold $K$ over the time horizon $[0,t]$, accounting for both the magnitude and duration of the threshold exceedance. Proposition \ref{Prop:integrated} shows that the non-diversification phenomenon of the underlying L\'{e}vy loss process is preserved for this cumulative excess-loss functional.  

\section{An example: diversification effect changes over time}\label{Sec:example}
For L\'{e}vy processes, the stochastic dominance at a fixed time horizon does not necessarily extend uniformly over time. More interestingly, the diversification effect may change with time horizon as demonstrated by the following proposition.
\begin{proposition}\label{ex:time-dependent-dominance}
Let $\{X_t\}_{t\ge0}$ be a compound Poisson process with zero
drift and L\'evy measure $\nu=\sum_{k\ge1}\lambda_k\delta_k$, where
$$
\lambda_k=\frac{1}{k!}
\left.\frac{d^k}{dz^k}\log Q(z)\right|_{z=0},\quad k\ge1,~ \text{with}~ Q(z)=\frac{2-\sqrt{1-z}}{2(1+\sqrt{1-z})},~ |z|\leq 1,
$$
and $\{X_t^{(1)}\}_{t\ge0}$ and $\{X_t^{(2)}\}_{t\ge0}$ be
independent copies of $\{X_t\}_{t\ge0}$. Then the following conclusions hold.
\begin{enumerate}[(i)]
\item For every $\varepsilon\in(0,1/2)$, there exists
$T_\varepsilon<\infty$ such that
$$
X_t\le_{\mathrm{st}}
\lambda X_t^{(1)}+(1-\lambda)X_t^{(2)}
\qquad\text{for all }t\ge T_\varepsilon
\text{ and all }\lambda\in[\varepsilon,1-\varepsilon].
$$
\item For all $\lambda\in (0,1)$ ,
$X_{1/10}\not\le_{\mathrm{st}}
\lambda X_{1/10}^{(1)}+(1-\lambda)X_{1/10}^{(2)}
$.
\end{enumerate}
\end{proposition}

For the compound Poisson process in Proposition~\ref{ex:time-dependent-dominance},
stochastic dominance fails at $t=1/10$ for every $\lambda\in(0,1)$,
but holds for all sufficiently large $t$ once $\lambda$ is fixed, demonstrating that  stochastic dominance may vary with the time horizon.
\section{Conclusion}\label{Sec:conc}
In this paper, we characterize the L\'{e}vy measures of L\'{e}vy processes exhibiting the non-diversification phenomenon or the reverse diversification order uniformly over all time horizons. For L\'{e}vy processes with bounded variation sample paths, we characterize the structure of the L\'{e}vy measures in terms of  subadditivity and concavity of the transformed L\'{e}vy tails. For general symmetric L\'{e}vy processes without a Gaussian component, we find that the symmetric $1$-stable process is the only nontrivial process exhibiting either phenomenon. 

We also study  multidimensional L\'{e}vy processes with heterogeneous and dependent components. For  L\'{e}vy processes with bounded variation sample paths, we characterize the corresponding L\'{e}vy measures. Without this restriction, we show that the L\'{e}vy measure of any symmetric L\'{e}vy process exhibiting either phenomenon is projection-invariant. For multidimensional $\alpha$-stable processes, although a similarly simple characterization is unavailable in general,  we derive  explicit characterizations of L\'{e}vy measures for the processes exhibiting the two adverse diversification phenomena when  the L\'{e}vy measures are supported on the positive orthant. Moreover, we characterize the L\'{e}vy measures for  L\'{e}vy processes with heterogeneous $\alpha$-stable components and   multidimensional compound Poisson processes.

Finally, we extend the  two adverse diversification phenomena  to  sample-path-dependent transformations of L\'{e}vy processes. We show that the non-diversification phenomenon is preserved for running maxima and integrals of increasing convex functionals, while both adverse diversification phenomena are preserved under L\'{e}vy-driven  stochastic integrals with nonnegative deterministic kernels. Applications to ruin theory, storage models and stochastic volatility are discussed. 

An interesting direction for future research is to obtain characterizations for general L\'{e}vy processes with infinite variation and for multidimensional L\'{e}vy processes with L\'{e}vy measures not restricted to particular orthants.  Moreover, further applications of uniform-in-time diversification properties to sample-path-dependent risk models also deserve investigation.

\section{Appendix}
In this appendix, we provide the proofs of all results and add one additional result for the multivariate $\alpha$-stable distributions.
\subsection{Proofs}
In this subsection, we provide all the proofs of the results in this paper.
Before we give our proofs, we first show an auxiliary result, which will play an important role in our proofs.
For a L\'{e}vy measure $\nu$, and $b\neq0$, let $T_b\nu(B)=\nu(b^{-1}B)$ for all $B\in\mathcal B(\R)$, where $\mathcal B(\R)$ denotes the collection of all Borel sets. We use the convention that $T_b\nu(B)=0$ for all  $B\in\mathcal B(\R)$ when $b=0$. Moreover, let $\nu^+(B)=\nu(B\cap(0,\infty))$ and $\nu^-(B)=\nu(B\cap(-\infty,0))$ for all $B\in\mathcal B(\R)$.  Applying Theorem 2.2 of \cite{ST93}, we arrive at the following result.
\begin{lemma}\label{Prop:comparison}
   Let $X_1$ and $X_2$ be infinitely divisible random variables with generating triplets $(0,\nu_i,\gamma_0^{(i)})$, $i\in [2]$, in the
uncompensated form, where $\int_D |x|\nu_i(dx)<\infty$, $i\in [2]$. 
If $\gamma_0^{(1)}\leq \gamma_0^{(2)}$,
$\nu_1((x,\infty))\leq \nu_2((x,\infty))$ on $(0,\infty)$ and $\nu_1((-\infty,x))\geq \nu_2((-\infty,x))$ on $(-\infty,0)$, then 
$X_1\leq_{\mathrm{st}}X_2$.
\end{lemma}
\begin{proof}
For $i\in [2]$, we can write $X_i\overset{d}{=}Y_i-Z_i$, where
$Y_i$ and $Z_i$ are independent infinitely divisible random variables
with uncompensated generating triplets $(0,\nu_i^+,\gamma_0^{(i)})$ and
$(0,T_{-1}\nu_i^-,0)$, respectively.
Note that  $\nu_1^+,\nu_2^+, T_{-1}\nu_1^{-}$ and $T_{-1}\nu_2^{-}$ are all supported on $(0,\infty)$. Clearly, for $x>0$, we have $\nu_1^+((x,\infty))\leq \nu_2^+((x,\infty))$ and $T_{-1}\nu_1^-((x,\infty))=\nu_1((-\infty,-x))\geq \nu_2((-\infty,-x))=T_{-1}\nu_2^-((x,\infty))$.
 By Theorem 2.2 of \cite{ST93}, we have $Y_1\leq_{\mathrm{st}}Y_2$ and $Z_2\leq_{\mathrm{st}}Z_1$. 
Using the independent assumption, we have $Y_1-Z_1\leq_{\mathrm{st}}Y_2-Z_2$, implying $X_1\leq_{\mathrm{st}}X_2$.  This completes the proof.
\end{proof}

{\it \bf Proof of Theorem \ref{th:main1}}. Let $\{X_t^{(1)}\}_{t\geq 0},\dots,\{X_t^{(n)}\}_{t\geq 0}$ be iid copies of $\{X_t\}_{t\geq 0}$. Note that if $\int_D|x|\nu(\d x)<\infty$, then for $t>0$, $X_t$ is infinitely divisible with generating triplet $(0, t\nu, t\gamma_0)$.
  For $t>0$, direct computation shows that \begin{align}\label{Eq:CF}\hat{\mu}_{\sum_{i=1}^n\theta_i X_t^{(i)}}(z)&=\exp\left[it\gamma_0z+\sum_{i=1}^nt\int_{\R} \left(e^{iz\theta_ix}-1 \right)\nu(\d x)\right]\nonumber\\
    &=\exp\left[it\gamma_0z+t\int_{\R} \left(e^{izx}-1 \right)\sum_{i=1}^n T_{\theta_i}\nu(\d x)\right],~z\in\R.
    \end{align}
    Hence $\sum_{i=1}^n\theta_i X_t^{(i)}$ follows an infinitely divisible distribution with generating triplet $(0,t\sum_{i=1}^n T_{\theta_i}\nu,t\gamma_0)$. The generating triplet of $X_t$ is $(0,t\nu,t\gamma_0)$. All generating triplets in this proof are written in the
uncompensated form. Since the weights sum to one, $X_t$ and
all weighted sums with weights in $\Delta_n$ have the same
uncompensated drift $t\gamma_0$.
    
We first focus on (i). We start with the `if' part.
Let $h_1(x)=\nu((x^{-1},\infty))$ for $x\in (0,\infty)$. Clearly, $h_1(0+)=0$. Then  $h_1$ is subadditive and continuous on $(0,\infty)$.  Using the subadditivity of $h_1$, we have, for $x\in (0,\infty)$,
    \begin{align}\label{Eq:h_1}\sum_{i=1}^nT_{\theta_i}\nu((x,\infty))&=\sum_{\theta_i>0}\nu((\theta_i^{-1}x,\infty))=\sum_{\theta_i> 0}h_1(\theta_i x^{-1})\nonumber\\
    &\geq h_1\left(x^{-1}\sum_{\theta_i>0}\theta_i\right)=\nu((x,\infty)).
    \end{align} 
    
Note also that $\sum_{i=1}^nT_{\theta_i}\nu((-\infty,x))=\nu((-\infty,x))=0$ for $x\in (-\infty,0)$. 
    Consequently, in light of Lemma \ref{Prop:comparison}, \eqref{Sup} holds for all $X_t$ with $t>0$ and all $\boldsymbol\theta\in \Delta_n$. This shows the ``if'' part.

 For the ``only if'' part, note that $X_t\leq_{\mathrm{st}} \theta_1X_t^{(1)}+\dots+\theta_nX_t^{(n)}$ for all $t>0$ implies that 
 $\mu_{X_t}((x,\infty))\leq \mu_{\sum_{i=1}^n\theta_iX_t^{(i)}}((x,\infty))$ for all $x>0$ and $t>0$. Using Corollary 8.9 of \cite{S99}, for $x>0$,  we have if $\nu(\{x\})=0$, then $$\lim_{t\downarrow 0}t^{-1}\mu_{X_t}((x,\infty))=\nu((x,\infty)),$$
 and if $\sum_{\theta_i>0}\nu(\{\theta_i^{-1}x\})=0$, then  
$$\lim_{t\downarrow 0}t^{-1}\mu_{\sum_{i=1}^n\theta_iX_t^{(i)}}((x,\infty))=\sum_{\theta_i>0}\nu((\theta_i^{-1}x,\infty)).$$
Hence, for $x>0$, if $\nu(\{x\})+\sum_{\theta_i>0}\nu(\{\theta_i^{-1}x\})=0$, then $\nu((x,\infty))\leq \sum_{\theta_i>0}\nu((\theta_i^{-1}x,\infty))$. Note that $\{x>0: \nu(\{x\})+\sum_{\theta_i>0}\nu(\{\theta_i^{-1}x\})>0\}$ is a countable set. Using the right continuity of $\nu((x,\infty))$, we have $\nu((x,\infty))\leq \sum_{\theta_i>0}\nu((\theta_i^{-1}x,\infty))$ holds for all $x>0$, implying that $\nu((x^{-1},\infty))$ is subadditive on $(0,\infty)$.

Moreover, note that  $t^{-1}\mu_{X_t}((-\infty,x))\geq t^{-1}\mu_{\sum_{i=1}^n\theta_iX_t^{(i)}}((-\infty,x))$ for all $x<0$ and $t>0$.  Letting $t\downarrow 0$, we can similarly obtain  $\nu((-\infty,x))\geq \sum_{\theta_i>0}\nu((-\infty,\theta_i^{-1}x))$ for all $x<0$, implying $\nu((-\infty,x^{-1}))$ is superadditive on $(-\infty,0)$.
 Let $h_2(x)=\nu((-\infty, x^{-1}))$ for $x\in (-\infty,0)$. Clearly, $h_2$ is superadditive and decreasing on $(-\infty,0)$ and $h_2(0-)=0$. Note that \begin{align*}
 \int_{-1}^0|x|\nu(\d x)&=\int_{-1}^0\nu([-1,x))\d x=\int_{-1}^0\nu((-\infty,x))\d x-\nu((-\infty,-1))\\
 &=\int_{-\infty}^{-1}x^{-2}h_2(x)\d x-\nu((-\infty,-1)).
 \end{align*}
 Hence,  $ \int_{-1}^0|x|\nu(\d x)<\infty$ is equivalent to $\int_{-\infty}^{-1}x^{-2}h_2(x)\d x<\infty$. Suppose there exists $x_0<0$ such that $h_2(x_0)>0$. Then for $x\in ((k+1)x_0,kx_0)$ with $k\in \mathbb N_+$, using the superadditivity and monotonicity of $h_2$, we have 
 $h_2(x)\geq (k-1)h_2(x_0)+h_2(x-(k-1)x_0)\geq kh_2(x_0)$. This implies
 \begin{align*}
  \int_{-\infty}^{-1}x^{-2}h_2(x)\d x\geq \sum_{k=k_0}^\infty \int_{(k+1)x_0}^{kx_0}x^{-2}kh_2(x_0)\d x=\sum_{k=k_0}^\infty \frac{-h_2(x_0)}{(k+1)x_0}=\infty,  
 \end{align*}
 where $k_0\geq -x_0^{-1}$, leading to a contradiction. Hence, $h_2(x)=0$ for all $x\in (-\infty,0)$.
We obtain the ``only if'' part.

Note that if $h_1$ is strictly subadditive, it follows from \eqref{Eq:h_1} that for $\boldsymbol\theta\in \Delta_n$ satisfying $\max_{i=1}^n\theta_i<1$, $\sum_{i=1}^nT_{\theta_i}\nu((x,\infty))>\nu((x,\infty))$ for all $x>0$, implying $\sum_{i=1}^nT_{\theta_i}\nu\neq \nu$. It follows from \eqref{Sup} and Theorem 8.1 of \cite{S99}  that \eqref{Sup} holds strictly for $t>0$.

 We next focus on (ii).  We start with the ``if'' part.
Recall that $h_1(x)=\nu((x^{-1},\infty))$. Then $h_1$ is  continuous and concave on $[0,\infty)$ with $h_1(0):=h_1(0+)=0$.
It follows that, for $x>0$ and  $\boldsymbol\theta,\boldsymbol\eta\in \Delta_n$ satisfying $\boldsymbol\theta \preceq\boldsymbol\eta$ ,
\begin{align}\label{Eq:h11}
    \sum_{i=1}^nT_{\theta_i}\nu((x,\infty))=\sum_{\theta_i>0}h_1(\theta_ix^{-1})~\text{and}~ \sum_{i=1}^nT_{\eta_i}\nu((x,\infty))=\sum_{\eta_i>0}h_1(\eta_ix^{-1}).
\end{align}
 For $x>0$, $\boldsymbol\theta, \boldsymbol\eta\in \Delta_n$ satisfying $\boldsymbol\theta\preceq \boldsymbol\eta$, we have $(\theta_1x^{-1},\dots,\theta_nx^{-1})\preceq (\eta_1x^{-1},\dots,\eta_nx^{-1})$. Using the concavity $h_1$ and $h_1(0)=0$, we have $\sum_{\theta_i>0}h_1(\theta_i x^{-1})\geq \sum_{\eta_i>0}h_1(\eta_i x^{-1})$, which implies $t\sum_{i=1}^nT_{\theta_i}\nu((x,\infty))\geq t\sum_{i=1}^nT_{\eta_i}\nu((x,\infty))$  for all  $x\in (0,\infty)$ and $t>0$.

Note also that $\sum_{i=1}^nT_{\theta_i}\nu((-\infty,x))=\sum_{i=1}^nT_{\eta_i}\nu((-\infty,x))=0$  for all  $x\in (-\infty,0)$.
    Consequently, it follows from Lemma \ref{Prop:comparison} that \eqref{Sub} holds for all $X_t$ with $t>0$ and all  $\boldsymbol\theta,\boldsymbol\eta\in \Delta_n$ satisfying $\boldsymbol\theta \preceq\boldsymbol\eta$ .
 
 Next, we focus on the ``only if'' part. For any  $\boldsymbol\theta,\boldsymbol\eta\in \Delta_n$ satisfying $\boldsymbol\theta \preceq\boldsymbol\eta$ ,  $\eta_1X_t^{(1)}+\dots+\eta_nX_t^{(n)}\leq_{\mathrm{st}} \theta_1X_t^{(1)}+\dots+\theta_nX_t^{(n)}$ for all $t>0$ implies that 
 $\mu_{\sum_{i=1}^n\eta_iX_t^{(i)}}((x,\infty))\leq \mu_{\sum_{i=1}^n\theta_iX_t^{(i)}}((x,\infty))$ for all $x\in\R$ and $t>0$. Using Corollary 8.9 of \cite{S99}, for $x>0$,  we have if $\sum_{\eta_i>0}\nu(\{\eta_i^{-1}x\})=0$,, then $$\lim_{t\downarrow 0}t^{-1}\mu_{\sum_{i=1}^n\eta_iX_t^{(i)}}((x,\infty))=\sum_{\eta_i>0}\nu((\eta_i^{-1}x,\infty)),$$
 and if $\sum_{\theta_i>0}\nu(\{\theta_i^{-1}x\})=0$, then  
$$\lim_{t\downarrow 0}t^{-1}\mu_{\sum_{i=1}^n\theta_iX_t^{(i)}}((x,\infty))=\sum_{\theta_i>0}\nu((\theta_i^{-1}x,\infty)).$$
Hence, for $x>0$, if $\sum_{\eta_i>0}\nu(\{\eta_i^{-1}x\})+\sum_{\theta_i>0}\nu(\{\theta_i^{-1}x\})=0$, then $\sum_{\eta_i>0}\nu((\eta_i^{-1}x,\infty))\leq \sum_{\theta_i>0}\nu((\theta_i^{-1}x,\infty))$. Using the right continuity of $\nu((x,\infty))$, we have $\sum_{\eta_i>0}\nu((\eta_i^{-1}x,\infty))\leq \sum_{\theta_i>0}\nu((\theta_i^{-1}x,\infty))$ for all $x>0$. 
Note that for $0<\lambda_1\leq \lambda_2\leq 1/2$, we have $(\lambda_2,1-\lambda_2,\dots,0)\preceq (\lambda_1,1-\lambda_1,\dots,0)$. Hence, we have $\nu((\lambda_2^{-1}x^{-1},\infty))+\nu((1-\lambda_2)^{-1}x^{-1},\infty))\geq \nu((\lambda_1^{-1}x^{-1},\infty)+\nu((1-\lambda_1)^{-1}x^{-1},\infty))$ for all $x>0$ and $0<\lambda_1\leq \lambda_2\leq 1/2$. Taking $\lambda_2=1/2$, $x=a+b$, and
$\lambda_1=\min\{a,b\}/(a+b)$ in the preceding inequality yields
$ 
2h_1\left(\frac{a+b}{2}\right)\geq h_1(a)+h_1(b),
  a,b>0.
$ 
Thus, $h_1$ is midpoint concave. Since $h_1$ is finite and
nondecreasing, and hence $\nu((x^{-1},\infty))$ is concave on $(0,\infty)$.

For $x<0$ and $t>0$, we have $t^{-1}\mu_{\sum_{i=1}^n\eta_iX_t^{(i)}}((-\infty,x))\geq t^{-1}\mu_{\sum_{i=1}^n\theta_iX_t^{(i)}}((-\infty,x))$. Letting $t\downarrow 0$, using Corollary 8.9 of \cite{S99}, we have 
$\sum_{\eta_i>0}\nu((-\infty, \eta_i^{-1}x))\geq \sum_{\theta_i>0}\nu((-\infty, \theta_i^{-1}x))$ for all $x<0$. This implies the convexity of $\nu((-\infty,x^{-1}))$ on $(-\infty,0)$. Let $h_2(x)=\nu((-\infty, x^{-1}))$ for $x\in (-\infty,0)$. Clearly, $h_2$ is convex and decreasing on $(-\infty,0)$ and $h_2(0-)=0$. Hence, we have, for $x,y\in (-\infty,0)$, $h_2(0-)-h_2(x)\geq h_2(y)-h_2(x+y)$, implying $h_2(x+y)\geq h_2(x)+h_2(y)$. This means that $h_2$ is superadditive on $(-\infty,0)$. Using the conclusion in the proof of (i), we have $h_2(x)=0$ for all $x\in (-\infty,0)$. We obtain the ``only if'' part. 

Finally, we show that if $h_1$ is nonlinear and concave on $(0,\infty)$, then  \eqref{Sub} holds strictly whenever $t>0$ and $\boldsymbol\theta \prec\boldsymbol\eta$.  Suppose, by contradiction, that \eqref{Sub} holds as an equality in distribution for some  $t_0>0$ and fixed $\boldsymbol\theta \prec\boldsymbol\eta$. We say $\boldsymbol\theta$ is a $T$-\emph{transform} of  $\boldsymbol\eta$ if for some $\lambda\in [0,1]$ and $i,j\in [n]$ with $i\neq j$, $\theta_i=\lambda\eta_i+(1-\lambda)\eta_j$, $\theta_j=(1-\lambda)\eta_i+\lambda\eta_j$, and $\theta_k=\eta_k$ for $k\in [n]\setminus\{i,j\}$. Using the result in Section 1.A.3 of \cite{MOA11}, $\boldsymbol\theta$ can be obtained from $\boldsymbol\eta$ by  a finite number of $T$-transforms, i.e., $\boldsymbol\theta=T_m\circ T_{m-1}\circ\dots\circ T_1(\boldsymbol\eta)$.  Then there exists $m_1$ such that $T_{m_1}\circ\dots\circ T_{1}(\boldsymbol\eta)\prec T_{m_1-1}\circ\dots\circ T_{1}(\boldsymbol\eta)$. One can easily check that \eqref{Sub} holds as an equality in distribution for  $t_0$ and $\boldsymbol\theta':=T_{m_1}\circ\dots\circ T_{1}(\boldsymbol\eta) \prec\boldsymbol\eta':=T_{m_1-1}\circ\dots\circ T_{1}(\boldsymbol\eta)$. As $\boldsymbol\theta' \prec\boldsymbol\eta'$, there exists  $\lambda\in (0,1)$ and $i,j\in [n]$ with $i\neq j$ satisfying $\eta_i'\neq \eta_j'$ such that  $\theta_i'=\lambda\eta_i'+(1-\lambda)\eta_j'$, $\theta_j'=(1-\lambda)\eta_i'+\lambda\eta_j'$, and $\theta_k'=\eta_k'$ for $k\in [n]\setminus\{i,j\}$. It follows from Theorem 8.1 of \cite{S99} that   $\sum_{i=1}^nT_{\theta_i'}\nu=\sum_{i=1}^nT_{\eta_i'}\nu$, which together with  \eqref{Eq:h11} implies $\sum_{i=1}^n h_1(\theta_i' x)=\sum_{i=1}^n h_1(\eta_i' x)$ for all $x>0$. Hence, we have $h_1((\lambda\eta_i'+(1-\lambda)\eta_j')x)+h_1(((1-\lambda)\eta_i'+\lambda\eta_j')x)=h_1(\eta_i'x)+h_1(\eta_j' x)$, which combined with  the concavity of $h_1$ implies $h_1((\lambda\eta_i'+(1-\lambda)\eta_j')x)=\lambda h_1(\eta_i'x)+(1-\lambda)h_1(\eta_j'x)$. It follows from the concavity of $h_1$ that $h_1$ is affine on $[\min(\eta_i',\eta_j') x,\max(\eta_i',\eta_j') x]$ for all $x>0$. Hence, $h_1$ is affine on $(0,\infty)$. Using the fact $h_1(0+)=0$, $h_1$ is  linear on $(0,\infty)$, leading to a contradiction. Hence, if $h_1$ is nonlinear and concave on $(0,\infty)$, then  \eqref{Sub} holds strictly whenever $t>0$ and $\boldsymbol\theta \prec\boldsymbol\eta$.   We complete the proof.
\qed

{\it\bf Proof of Proposition \ref{Prop:onlyif}}. The proof of this proposition is given in the proof of the ``only if'' parts of Theorem \ref{th:main1}. \qed

{\it\bf Proof of Proposition \ref{Prop:0}}.
It follows from Theorem 25.3 of \cite{S99} that $\mathbb E(|X|)=\infty$ if and only if $\int_{|x|>1} |x|\nu(\d x)=\infty$, which is equivalent to $\int_1^\infty \nu((x,\infty))\d x=\infty$. Let $h(x)=\nu((x^{-1},\infty))$ on $(0,\infty)$. Clearly, $h(0+)=0$ and $h$ is subadditive and increasing on $(0,\infty)$, which implies that $h$ is continuous on $(0,\infty)$. If $h(x_0)=0$ for some $x_0>0$, then using the subadditivity, we have $h(x)=0$ for all $x\in (0,\infty)$, which contradicts $\lim_{x\to\infty}h(x)=\nu((0,\infty))>0$. Hence $h(x)>0$ for all $x\in (0,\infty)$.  For $x\in ((k+1)^{-1}, k^{-1}]$ with $k\in \mathbb N_+$, we have $h(1)\leq h((k+1)x)\leq (k+1)h(x)$, implying $h(x)\geq (k+1)^{-1}h(1)$.  Direct computation shows
\begin{align*}\int_1^\infty \nu((x,\infty))\d x&=\int_0^1 x^{-2}h(x)\d x=\sum_{k=1}^\infty\int_{(k+1)^{-1}}^{k^{-1}}x^{-2}h(x)\d x\\
&\geq \sum_{k=1}^\infty h(1)(k+1)^{-1}\int_{(k+1)^{-1}}^{k^{-1}}x^{-2}\d x\\
&=\sum_{k=1}^\infty h(1)(k+1)^{-1}=\infty.
\end{align*}
Hence, $\mathbb E(|X|)=\infty$. \qed

 {\it \bf Proof of Proposition \ref{Prop:2}}. Let $\{X_t^{(1)}\}_{t\geq 0},\dots,\{X_t^{(n)}\}_{t\geq 0}$ be iid copies of $\{X_t\}_{t\geq 0}$.
 First, we consider (i). For $\boldsymbol\theta\in\Delta_n$, using the subadditivity of $g$, we have
    \begin{align*}\sum_{\theta_i>0}\nu((\theta_i^{-1}x^{-1},\infty))&=\sum_{\theta_i>0}\int_{\theta_i^{-1}x^{-1}}^\infty t^{-1}g(t^{-1})\d t
    =\int_{x^{-1}}^\infty t^{-1}\sum_{\theta_i>0}g(\theta_it^{-1})\d t\\
    &\geq \int_{x^{-1}}^\infty t^{-1}g(t^{-1})\d t=\nu((x^{-1},\infty)),~x>0.
    \end{align*}
    Hence, $\nu((x^{-1},\infty))$ is subadditive on $(0,\infty)$.
    It follows from (i) of Theorem \ref{th:main1} that \eqref{Sup} holds.

    Next, we suppose  $g$ is strictly subadditive on $(0,\infty)$.  By \eqref{Eq:CF}, $\sum_{i=1}^n\theta_i X_t^{(i)}$ follows an infinitely divisible distribution with generating triplet $(0,t\sum_{i=1}^n T_{\theta_i}\nu,t\gamma_0)$.  For  $\boldsymbol\theta\in \Delta_n$ satisfying  $\max_{i=1}^n\theta_i<1$ and a Borel measurable set $B\subset (0,\infty)$ with a positive Lebesgue measure, we have 
    \begin{align*}\sum_{i=1}^nT_{\theta_i}\nu(B)&=\sum_{\theta_i>0} \nu(\theta_i^{-1}B)=\sum_{\theta_i>0}\int_{\theta_i^{-1}B}x^{-1}g(x^{-1})\d x\\
    &=\int_{B}\sum_{\theta_i>0}x^{-1}g(\theta_ix^{-1})\d x>\int_{B}x^{-1}g(x^{-1})\d x\\
    &=\nu(B).
    \end{align*}
     Denote by $\nu_1=\sum_{i=1}^nT_{\theta_i}\nu-\nu\geq 0$. There exist independent  L\'{e}vy processes $\{Y_t\}_{t\geq 0}$ and  $\{Z_t\}_{t\geq 0}$ with generating triplets $(0,\nu_1,0)$ and  $(0,v,\gamma_0)$ respectively. Hence, $\sum_{i=1}^n\theta_iX_t^{(i)}\overset{d}{=}Y_t+Z_t$ and $X_t\overset{d}{=}Z_t$, where $\overset{d}{=}$ means equality in distribution.
    
    By $\nu((-\infty,0))=0$ and $\int_{D}|x|\nu(\d x)<\infty$, we have  $\nu_1((-\infty,0))=0$ and  $\int_{D}|x|\nu_1(\d x)<\infty$.   Note also that $\nu_1\neq 0$. Applying Theorem 19.3 of \cite{S99}, we have  $Y_t\geq 0$ a.s. and $\mathbb P(Y_t>0)>0$ for $t>0$. In light of Theorem 24.3 of \cite{S99}, the support of $Y_t$ with $t>0$ is unbounded from above. For $x\in (\essinf X_t, \esssup X_t)$, there exists $a>0$ such that $F_{X_t}(x-a)<F_{X_t}(x)$ and $F_{Y_t}(a)<1$. Hence, for $x\in (\essinf X_t, \esssup X_t)$, we have
    $$F_{X_t}(x)-F_{\sum_{i=1}^n\theta_iX_t^{(i)}}(x)=\int_{[0,\infty)} (F_{X_t}(x)-F_{X_t}(x-y))\d F_{Y_t}(y)>0.$$
    This implies that \eqref{Sup} is strongly strict. We establish claim (i).

    Next, we focus on (ii).  If $g$ is nonnegative and concave on $(0,\infty)$, then $g(0+)$ exists and is nonnegative. This means that $g$ is nonnegative and concave on $[0,\infty)$ with $g(0)=g(0+)\geq 0$. For $x\in (0,\infty)$, $\boldsymbol\theta, \boldsymbol\eta\in \Delta_n$ satisfying $\boldsymbol\theta\preceq \boldsymbol\eta$, we have $(\theta_1x,\dots,\theta_nx)\preceq (\eta_1x,\dots,\eta_nx)$. Using the  concavity $g$, we have $\sum_{\theta_i>0} g(\theta_i x)+k_1g(0)\geq \sum_{\eta_i>0} g(\eta_i x)+k_2g(0)$, where $k_1=\#\{i\in [n]:\theta_i=0\}$ and $k_2=\#\{i\in [n]: \eta_i=0\}$. By definition, we have $ \sum_{i=1}^{k_1}\eta_{(i)}\leq \sum_{i=1}^{k_1} \theta_{(i)}=0$, implying $k_2\geq k_1$.  Consequently, we have $\sum_{\theta_i>0} g(\theta_i x)\geq \sum_{\eta_i>0} g(\eta_i x)$ for all $x\in (0,\infty)$.   This implies that, for $x\in (0,\infty)$, $\boldsymbol\theta, \boldsymbol\eta\in \Delta_n$ satisfying $\boldsymbol\theta\preceq \boldsymbol\eta$,  
    \begin{align*}\sum_{\theta_i>0}\nu((\theta_i^{-1}x^{-1},\infty))&=\sum_{\theta_i>0}\int_{\theta_i^{-1}x^{-1}}^\infty t^{-1}g(t^{-1})\d t
    =\int_{x^{-1}}^\infty t^{-1}\sum_{\theta_i>0}g(\theta_it^{-1})\d t\\
    &\geq \int_{x^{-1}}^\infty t^{-1}\sum_{\eta_i>0}g(\eta_it^{-1})\d t=\sum_{\eta_i>0}\nu((\eta_i^{-1}x^{-1},\infty)).
    \end{align*}
    Note that for $0<\lambda_1\leq \lambda_2\leq 1/2$, we have $(\lambda_2,1-\lambda_2,\dots,0)\preceq (\lambda_1,1-\lambda_1,\dots,0)$. Hence, we have $\nu((\lambda_2^{-1}x^{-1},\infty)+\nu((1-\lambda_2)^{-1}x^{-1},\infty))\geq \nu((\lambda_1^{-1}x^{-1},\infty)+\nu((1-\lambda_1)^{-1}x^{-1},\infty))$ for all $x>0$ and $0<\lambda_1\leq \lambda_2\leq 1/2$, which implies that $\nu((x^{-1},\infty))$ is concave on $(0,\infty)$. It follows from  Theorem \ref{th:main1} that \eqref{Sub} holds. 
    
    Next, we suppose  $g$ is strictly concave.  By \eqref{Eq:CF}, the generating triplets of $\sum_{i=1}^n\theta_i X_i$ and  $\sum_{i=1}^n\eta_i X_i$ are $(0,\sum_{i=1}^n T_{\theta_i}\nu,\gamma_0)$ and $(0,\sum_{i=1}^n T_{\eta_i}\nu,\gamma_0)$ respectively.  For  $\boldsymbol\theta \prec\boldsymbol\eta$ with $\boldsymbol\theta,\boldsymbol\eta\in \Delta_n$ and a Borel measurable set $B\subset (0,\infty)$ with a positive Lebesgue measure, we have 
    \begin{align*}\sum_{i=1}^nT_{\theta_i}\nu(B)&=\sum_{\theta_i>0} \nu(\theta_i^{-1}B)=\sum_{\theta_i>0}\int_{\theta_i^{-1}B}x^{-1}g(x^{-1})\d x\\
    &=\int_{B}\sum_{\theta_i>0}x^{-1}g(\theta_ix^{-1})\d x>\int_{B}\sum_{\eta_i>0}x^{-1}g(\eta_ix^{-1})\d x\\
    &=\sum_{i=1}^nT_{\eta_i}\nu(B).
    \end{align*}
     Denote by $\nu_2=\sum_{i=1}^nT_{\theta_i}\nu-\sum_{i=1}^nT_{\eta_i}\nu\geq 0$. Clearly, $\nu_2\neq 0$.
    By $\nu((-\infty,0))=0$ and $\int_{D}|x|\nu(\d x)<\infty$, we have  $\nu_2((-\infty,0))=0$ and  $\int_{D}|x|\nu_2(\d x)<\infty$.   Using the same argument as in the proof of (i), we can show the strong strictness of \eqref{Sub}.  We complete the proof. \qed

    {\it\bf Proof of Theorem \ref{Prop:S}}. We first consider (i). Suppose  \eqref{Sup} holds for all $X_t$ with $t>0$ and all $\boldsymbol\theta\in \Delta_n$. Then it follows from Proposition \ref{Prop:onlyif} that $\nu((x^{-1},\infty))$ is subadditive on $(0,\infty)$ and $\nu((-\infty,x^{-1}))$ is superadditive on $(-\infty,0)$. As $\nu$ is symmetric, we have 
$\nu((x^{-1},\infty))=\nu((-\infty,(-x)^{-1}))$ for $x\in (0,\infty)$. Hence, $\nu((x^{-1},\infty))$ is superadditive as well, which further implies that $\nu((x^{-1},\infty))$ is a linear function. Note that $\lim_{x\downarrow 0}\nu((x^{-1},\infty))=0$. Hence, we have $\nu((x^{-1},\infty))=cx$ with some $c\geq 0$ for $x\in(0,\infty)$. It follows that $\nu(\d x)=cx^{-2}\d x$ on $(0,\infty)$. Using the symmetry of $\nu$, we have $\nu(\d x)=c|x|^{-2}\d x$ on $\R\setminus\{0\}$.

If $\nu(\d x)=c|x|^{-2}\d x$ on $\R\setminus\{0\}$, we have $\hat{\mu}_{X_t}(z)=e^{-\pi ct|z|}$ for $z\in\R$; see e.g., Theorem 14.15 of \cite{S99}.  One can easily check that $X_t\overset{d}{=}\theta_1X_t^{(1)}+\dots+\theta_nX_t^{(n)}$, where $X_t^{(i)},~i\in[n]$ are iid copies of $X_t$.  Hence, \eqref{Sup} holds. We establish the claim in (i).

The proof of (ii) is exactly the same as that of (i). Hence, it is omitted. We complete the proof. \qed

{\it \bf Proof of Proposition \ref{th:main2}}.
  Note that for any $\boldsymbol{\theta}\in\Delta_n$, $\{\sum_{i=1}^n\theta_iX_t^{(i)}\}_{t\geq 0}$ is a one dimensional L\'{e}vy process with generating triplet  $(0,T_{\boldsymbol\theta}\nu,\gamma_0)$. Moreover, it follows that 
  $$\int_{\R}|x|\wedge 1T_{\boldsymbol\theta}\nu(\d x)=\int_{\R^n}\left|\sum_{i=1}^n\theta_ix_i\right|\wedge 1\nu(\d \mathbf x)\leq n\int_{\R^n}\|\mathbf x\|\wedge 1\nu(\d \mathbf x)<\infty.$$
Hence, the ``if'' parts in (i) and (ii) follow  directly from Lemma \ref{Prop:comparison}. 

Next, we show the ``only if'' parts.   Note that, for  $t>0$, $X_t^{(1)}\leq_{\mathrm{st}}\sum_{i=1}^n\theta_iX_t^{(i)}$ is equivalent to $\mu_{X_t^{(1)}}((u,\infty))\leq \mu_{\sum_{i=1}^n\theta_iX_t^{(i)}}((u,\infty))$ on $(0,\infty)$ and  $\mu_{X_t^{(1)}}((-\infty,u ))\geq \mu_{\sum_{i=1}^n\theta_iX_t^{(i)}}((-\infty,u ))$ on $(-\infty,0)$. Then the  inequalities on L\'{e}vy measures can be derived similarly as in the proof of Theorem \ref{th:main1}. We next give some  details.  Using Corollary 8.9 of \cite{S99}, for $u>0$,  we have if $\nu(\{\mathbf x\in\R^n: x_1=u\})=0$, then $$\lim_{t\downarrow 0}t^{-1}\mu_{X_t^{(1)}}((u,\infty))=\nu(\{\mathbf x\in\R^n: x_1>u\}),$$
 and if $\nu(\{\mathbf x\in\R^n: \sum_{i=1}^n\theta_ix_i=u\})=0$, then  
$$\lim_{t\downarrow 0}t^{-1}\mu_{\sum_{i=1}^n\theta_iX_t^{(i)}}((u,\infty))=\nu\left(\left\{\mathbf x\in\R^n: \sum_{i=1}^n\theta_ix_i>u\right\}\right).$$
Hence, we have $\nu(\{\mathbf x\in\R^n: \sum_{i=1}^n\theta_ix_i>u\})\geq \nu(\{\mathbf x\in\R^n: x_1>u\})$ if
$\nu(\{\mathbf x\in\R^n: x_1=u\})+\nu(\{\mathbf x\in\R^n: \sum_{i=1}^n\theta_ix_i=u\})=0$. Note that  $\nu(\{\mathbf x\in\R^n: x_1=u\})+\nu(\{\mathbf x\in\R^n: \sum_{i=1}^n\theta_ix_i=u\})>0$ holds on a countable number of $u$.
 Using the right continuity of the corresponding functions, we have $\nu(\{\mathbf x\in\R^n: \sum_{i=1}^n\theta_ix_i>u\})\geq \nu(\{\mathbf x\in\R^n: x_1>u\})$ holds for all $u>0$. Applying Corollary 8.9 of \cite{S99} again and using the same argument as above, we have $\nu(\{\mathbf x\in\R^n: \sum_{i=1}^n\theta_ix_i<u\})\leq \nu(\{\mathbf x\in\R^n: x_1<u\})$ on $(-\infty,0)$. We establish the ``only if'' part of (i). The L\'{e}vy measure inequality in (ii) can be shown similarly and the details are omitted. 
 
 The strictness statements follow immediately from the uniqueness of the generating triplet, since the corresponding infinitely divisible distributions have different L\'{e}vy measures. We complete the proof. \qed

{\bf Proof of Theorem \ref{prop:symmetric}}. (i) $\Rightarrow$ (iii). First, suppose \eqref{Supm} holds for all $\mathbf X_t$ with $t>0$ and all $\boldsymbol\theta\in\Delta_n$.
    Note that the ``only if'' parts of Proposition \ref{th:main2} hold without the constraint $\int_{\R^n}\|\mathbf x\|\wedge 1\nu(\d \mathbf x)<\infty$. Hence, we have $\nu(\{\mathbf x\in\R^n: \sum_{i=1}^n\theta_ix_i>u\})\geq \nu(\{\mathbf x\in\R^n: x_1>u\})$ on $(0,\infty)$ and $\nu(\{\mathbf x\in\R^n: \sum_{i=1}^n\theta_ix_i<u\})\leq \nu(\{\mathbf x\in\R^n: x_1<u\})$ on $(-\infty,0)$.  The symmetry of $\mathbf X_t$ implies the symmetry of $\nu$. This implies $\nu(\{\mathbf x\in\R^n: \sum_{i=1}^n\theta_ix_i>u\})\leq \nu(\{\mathbf x\in\R^n: x_1>u\})$ on $(0,\infty)$ and $\nu(\{\mathbf x\in\R^n: \sum_{i=1}^n\theta_ix_i<u\})\geq \nu(\{\mathbf x\in\R^n: x_1<u\})$ on $(-\infty,0)$. Consequently, $\nu(\{\mathbf x\in\R^n: \sum_{i=1}^n\theta_ix_i>u\})=\nu(\{\mathbf x\in\R^n: x_1>u\})$ on $\R$ for all $\boldsymbol\theta\in\Delta_n$. This implies $T_{\boldsymbol\theta}\nu=T_{\boldsymbol\eta}\nu$ for all $\boldsymbol\theta, \boldsymbol\eta\in \Delta_n$.  

    (iii) $\Rightarrow$ (i) and (ii). Direct computation shows that, for $t>0$,
    \begin{align*}\log\hat{\mu}_{\sum_{i=1}^n \theta_iX_t^{(i)}}(z)&=t\int_{\R^n} \left(e^{iz\langle\boldsymbol \theta,\mathbf x\rangle}-1-iz\langle\boldsymbol\theta,\mathbf x\rangle\id_{D}(\mathbf x) \right)\nu(\d\mathbf x)\\
    &=t\int_{\R^n} \left(e^{iz\langle\boldsymbol \theta,\mathbf x\rangle}-1-iz\langle\boldsymbol\theta,\mathbf x\rangle\id_{\{\mathbf x: |\langle\boldsymbol\theta,\mathbf x\rangle|\leq 1\}} \right)\nu(\d\mathbf x)\\
    &\quad+izt\int_{\R^n} \langle\boldsymbol\theta,\mathbf x\rangle \left(\id_{\{\mathbf x: |\langle\boldsymbol\theta,\mathbf x\rangle|\leq 1\}}-\id_{\{\mathbf x: \|\mathbf x\|\leq 1\}} \right)\nu(\d\mathbf x)\\
    &=t\int_{\R} \left(e^{izx}-1-izx\id_{[-1,1]}(x) \right)T_{\boldsymbol\theta}\nu(\d x),~z\in\R.
    \end{align*}
    The last equality is because $\int_{\R^n} \langle\boldsymbol\theta,\mathbf x\rangle \left(\id_{\{\mathbf x: |\langle\boldsymbol\theta,\mathbf x\rangle|\leq 1\}}-\id_{\{\mathbf x: \|\mathbf x\|\leq 1\}} \right)\nu(\d\mathbf x)=0$ if $\nu$ is symmetric, which is implied by the symmetry of $\mathbf X_t$.  Hence, $T_{\boldsymbol\theta}\nu=T_{\boldsymbol\eta}\nu$ for all $\boldsymbol\theta, \boldsymbol\eta\in \Delta_n$ implies 
    $\sum_{i=1}^n \theta_iX_t^{(i)}\overset{d}{=}\sum_{i=1}^n \eta_iX_t^{(i)}$ for all $\boldsymbol\theta, \boldsymbol\eta\in \Delta_n$.  We establish the stochastic dominance in \eqref{Supm} and \eqref{Subm}. Note that (ii) $\Rightarrow$ (i) holds obviously.  We complete the proof.
\qed

  {\it \bf Proof of Proposition \ref{Prop:SM1}}. Clearly, the symmetry of $\mathbf X_t$ implies the symmetry of $\nu$ and $\rho$.  Applying Theorem \ref{prop:symmetric}, we only need to show that $T_{\boldsymbol\theta}\nu=T_{\boldsymbol\eta}\nu$ for all $\boldsymbol\theta, \boldsymbol\eta\in \Delta_n$ is equivalent to (iii) of Proposition \ref{Prop:SM1}. 
Using \eqref{stablev}, we have $T_{\boldsymbol\theta}\nu((x,\infty))=x^{-1}\int_{S_n}(\sum_{i=1}^n\theta_i\xi_i)_+\rho(\d \boldsymbol\xi)$ for $x>0$ and $T_{\boldsymbol\theta}\nu((-\infty,x))=|x|^{-1}\int_{S_n}(\sum_{i=1}^n\theta_i\xi_i)_-\rho(\d \boldsymbol\xi)$ for $x<0$. 

Hence, if (iii) holds, then $T_{\boldsymbol\theta}\nu((x,\infty))=x^{-1}\int_{S_n^+}\sum_{i=1}^n\theta_i\xi_i\rho(\d \boldsymbol\xi)=x^{-1}\int_{S_n}(\xi_i)_+\rho(\d \boldsymbol\xi)$ for $x>0$ and $T_{\boldsymbol\theta}\nu((-\infty,x))=|x|^{-1}\int_{S_n^-}\sum_{i=1}^n\theta_i(\xi_i)_-\rho(\d \boldsymbol\xi)=|x|^{-1}\int_{S_n}(\xi_i)_-\rho(\d \boldsymbol\xi)$ for $x<0$. Those expressions are independent of $\boldsymbol\theta\in \Delta_n$. Consequently, $T_{\boldsymbol\theta}\nu=T_{\boldsymbol\eta}\nu$ for all $\boldsymbol\theta, \boldsymbol\eta\in \Delta_n$.

Conversely, suppose $T_{\boldsymbol\theta}\nu=T_{\boldsymbol\eta}\nu$ for all $\boldsymbol\theta, \boldsymbol\eta\in \Delta_n$. Recall that the canonical basis vectors of $\mathbb R^n$ are denoted by $\boldsymbol e_i,~i\in [n]$. It follows that $T_{\boldsymbol e_i}((x,\infty))=x^{-1}\int_{S_n}(\xi_i)_+\rho(\d \boldsymbol\xi)$ for $x>0$ and $T_{\boldsymbol e_i}((-\infty, x))=|x|^{-1}\int_{S_n}(\xi_i)_-\rho(\d \boldsymbol\xi)$ for $x<0$. Hence, $\int_{S_n}(\xi_i)_+\rho(\d \boldsymbol\xi),i\in [n]$ are all equal, and $ \int_{S_n}(\xi_i)_-\rho(\d \boldsymbol\xi),~i\in[n]$ are all equal. Using the symmetry of $\rho$, we have $\int_{S_n}(\xi_i)_+\rho(\d \boldsymbol\xi), \int_{S_n}(\xi_i)_-\rho(\d \boldsymbol\xi),~i\in[n]$ are all equal.

Moreover, using the convexity of function $x_+$, we have, for $x>0$, \begin{align*}
    xT_{\boldsymbol\theta}\nu((x,\infty))&=\int_{S_n}(\sum_{i=1}^n\theta_i\xi_i)_+\rho(\d \boldsymbol\xi)\leq \int_{S_n}\sum_{i=1}^n\theta_i(\xi_i)_+\rho(\d \boldsymbol\xi)\\
    &=\int_{S_n}(\xi_1)_+\rho(\d \boldsymbol\xi).
\end{align*}
Note that $xT_{\boldsymbol\theta}\nu((x,\infty))= xT_{\boldsymbol e_1}\nu((x,\infty))=\int_{S_n}(\xi_1)_+\rho(\d \boldsymbol\xi)$ for $x>0$. Hence, $\int_{S_n}(\sum_{i=1}^n\theta_i\xi_i)_+\rho(\d \boldsymbol\xi)=\int_{S_n}\sum_{i=1}^n\theta_i(\xi_i)_+\rho(\d \boldsymbol\xi)$.
Note that if $\theta_i>0$ for all $i\in [n]$, then $(\sum_{i=1}^n\theta_i\xi_i)_+<\sum_{i=1}^n\theta_i(\xi_i)_+$ over $S_n\setminus (S_n^+\cup S_n^-)$ and $(\sum_{i=1}^n\theta_i\xi_i)_+=\sum_{i=1}^n\theta_i(\xi_i)_+$ over $(S_n^+\cup S_n^-)$.
Hence,  the above equality implies $\rho(S_n\setminus (S_n^+\cup S_n^-))=0$. We complete the proof. \qed

{\it \bf Proof of Theorem \ref{Thm:4}}. For $\boldsymbol\theta\in\Delta_n$, let $B_{u,\boldsymbol\theta}=\left\{\mathbf x\in\R^n: \sum_{i=1}^n\theta_ix_i>u\right\}$. Then for $u>0$,
   a direct computation shows that 
    \begin{align}\label{Eq:alpha}
        \nu(B_{u,\boldsymbol\theta})&=\int_{S_n^+}\int_0^\infty r^{-1-\alpha}\id_{B_{u,\boldsymbol\theta}}(r\boldsymbol \xi)\d r\rho(\d \boldsymbol \xi)\nonumber\\
        &=\int_{S_n^+}\int_{u(\sum_{i=1}^n\theta_i\xi_i)^{-1}}^\infty r^{-1-\alpha}\d r\rho(\d \boldsymbol \xi)\nonumber\\
        &=   \frac{ u^{-\alpha}}\alpha\int_{S_n^+}\left(\sum_{i=1}^n\theta_i\xi_i\right)^{\alpha}\rho(\d \boldsymbol \xi).
    \end{align}

    We first focus on (i).  Note that $\nu(\{\mathbf x\in\R^n: \sum_{i=1}^n\theta_ix_i<u\})=\nu(\{\mathbf x\in\R^n: x_1<u\})=0$ for all $u\in (-\infty,0)$. 
    Hence, applying Proposition \ref{th:main2} and using \eqref{Eq:alpha}, we have \eqref{Supm} holds for all $\mathbf X_t$ with $t>0$ and all $\boldsymbol\theta\in\Delta_n$ if and only if  $\int_{S_n^+}\left(\sum_{i=1}^n\theta_i\xi_i\right)^{\alpha}\rho(\d \boldsymbol \xi)\geq \int_{S_n^+}\xi_1^{\alpha}\rho(\d \boldsymbol \xi)$ for all $\boldsymbol\theta\in\Delta_n$.
    
    If $\min_{i\in [n]}\int_{S_n^+}\xi_i^\alpha\rho(\d \boldsymbol\xi)\geq \int_{S_n^+}\xi_1^\alpha\rho(\d \boldsymbol\xi)$, using the concavity of $x^\alpha$, we have
    \begin{align*}\int_{S_n^+}\left(\sum_{i=1}^n\theta_i\xi_i\right)^{\alpha}\rho(\d \boldsymbol \xi)\geq\int_{S_n^+}\sum_{i=1}^n\theta_i\xi_i^{\alpha}\rho(\d \boldsymbol \xi)\geq \min_{i=1}^n\int_{S_n^+}\xi_i^{\alpha}\rho(\d \boldsymbol \xi)
    =\int_{S_n^+}\xi_1^{\alpha}\rho(\d \boldsymbol \xi).
    \end{align*}
   If $\int_{S_n^+}\left(\sum_{i=1}^n\theta_i\xi_i\right)^{\alpha}\rho(\d \boldsymbol \xi)\geq \int_{S_n^+}\xi_1^\alpha\rho(\d \boldsymbol\xi)$ for all $\boldsymbol\theta\in\Delta_n$, then $\min_{i\in[n]}\int_{S_n^+}\xi_i^\alpha\rho(\d \boldsymbol\xi)=\int_{S_n^+}\xi_1^\alpha\rho(\d \boldsymbol\xi)$. 
    We establish (i).

    Next, we consider (ii). Note that $\nu(\{\mathbf x\in\R_+^n: \sum_{i=1}^n\theta_ix_i<u\})=\nu(\{\mathbf x\in\R_+^n: \sum_{i=1}^n\eta_ix_i<u\})=0$ for $u\in (-\infty,0)$. Hence, in light of Proposition \ref{th:main2} and \eqref{Eq:alpha}, the stochastic dominance in \eqref{Subm} holds for all $\mathbf X_t$ with $t>0$ and all   $\boldsymbol\theta,\boldsymbol\eta\in \Delta_n$ satisfying $\boldsymbol\theta \preceq\boldsymbol\eta$  if and only if $\int_{S_n^+}\left(\sum_{i=1}^n\theta_i\xi_i\right)^{\alpha}\rho(\d \boldsymbol \xi)\geq \int_{S_n^+}\left(\sum_{i=1}^n\eta_i\xi_i\right)^{\alpha}\rho(\d \boldsymbol \xi)$ for all  $\boldsymbol\theta,\boldsymbol\eta\in \Delta_n$ satisfying $\boldsymbol\theta \preceq\boldsymbol\eta$ .

    Let $H(\boldsymbol\theta)=\int_{S_n^+}\left(\sum_{i=1}^n\theta_i\xi_i\right)^{\alpha}\rho(\d \boldsymbol \xi)$ for all $\boldsymbol\theta\in \Delta_n$. The permutation invariance of $\rho$ implies that $H(\boldsymbol\theta)$ is permutation invariant on $\Delta_n$. Due to $\alpha\in (0,1)$, we have that $H(\boldsymbol\theta)$ is concave on $\Delta_n$. By Theorem 2.A.2 of \cite{MOA11},  for any $\boldsymbol\theta \preceq\boldsymbol\eta$, there exist $w_i\geq 0, i\in [n!]$ satisfying $\sum_{i=1}^{n!}w_i=1$  such that $\boldsymbol\theta= \sum_{k=1}^{n!} w_k T_{\sigma_k}\boldsymbol\eta$, where $T_{\sigma_k}$ represent all the possible permutation transforms. Hence, using the concavity and permutation invariance of $H$, we have $H(\boldsymbol\theta)=H(\sum_{k=1}^{n!} w_k T_{\sigma_k}\boldsymbol\eta)\geq H(\boldsymbol\eta)$.

    Next, suppose $H(\boldsymbol\theta)\geq H(\boldsymbol\eta)$ for all  $\boldsymbol\theta,\boldsymbol\eta\in \Delta_n$ satisfying $\boldsymbol\theta \preceq\boldsymbol\eta$ . We will show that $\rho$ is permutation-invariant. Clearly, for $\boldsymbol\theta\in \Delta_n$ and a permutation transform $T_{\sigma}$, we have $$\int_{S_n^+}\left(\sum_{i=1}^n\theta_i\xi_i\right)^\alpha\rho(\d \boldsymbol\xi)=H(\boldsymbol\theta)=H(T_{\sigma}(\boldsymbol\theta))=\int_{S_n^+}\left(\sum_{i=1}^n\theta_i\xi_i\right)^\alpha\rho_{\sigma}(\d \boldsymbol\xi),$$
    where $\rho_\sigma(B)=\rho(T_\sigma^{-1} B)$ for $B\in\mathcal B(S)$ with $T_\sigma^{-1}$ representing the inverse of $T_\sigma$. One can easily check that $H(\boldsymbol\theta)=H(T_{\sigma}(\boldsymbol\theta))$
    for all $\boldsymbol\theta\in \R_+^n$. Using the fact that for $x>0$ and $\alpha\in (0,1)$,
    $x^{\alpha}=\frac{\alpha}{\Gamma(1-\alpha)}\int_0^\infty (1-e^{-rx})r^{-1-\alpha}\d r$, we have 
    \begin{align*}
H(\boldsymbol\theta)&=\frac{\alpha}{\Gamma(1-\alpha)}\int_{S_n^+}\int_0^\infty (1-e^{-r\sum_{i=1}^n\theta_i\xi_i})r^{-1-\alpha}\d r\rho(\d\boldsymbol\xi),\\
H(T_\sigma\circ\boldsymbol\theta)&=\frac{\alpha}{\Gamma(1-\alpha)}\int_{S_n^+}\int_0^\infty (1-e^{-r\sum_{i=1}^n\theta_i\xi_i})r^{-1-\alpha}\d r\rho_\sigma(\d\boldsymbol\xi).
    \end{align*}
    Then $H(\boldsymbol\theta)=-\frac{\alpha}{\Gamma(1-\alpha)}\ln\E e^{-\sum_{i=1}^n\theta_iX_i}$, where $\mathbf X$ has an infinitely divisible distribution with triplet $(0,\nu,0)$, where $\nu$ is defined by \eqref{stablev}. Similarly, $H(T_\sigma\circ\boldsymbol\theta)=-\frac{\alpha}{\Gamma(1-\alpha)}\ln\E e^{-\sum_{i=1}^n\theta_iY_i}$, where $\mathbf Y$ has an infinitely divisible distribution with triplet $(0,\nu_\sigma,0)$ with $\nu_\sigma$  given by 
    $$\nu_\sigma(B)=\int_{S_n^+}\int_0^\infty r^{-1-\alpha}\id_B(r\boldsymbol \xi)\d r\rho_\sigma(\d \boldsymbol \xi).$$
    Hence, we have $\E e^{-\sum_{i=1}^n\theta_iX_i}=\E e^{-\sum_{i=1}^n\theta_iY_i}$ for all $\boldsymbol\theta\in \R_+^n$. Note that both $\mathbf X$ and $\mathbf Y$ are non-negative. Therefore, we have $\mathbf X\overset{d}{=}\mathbf Y$. It follows from Theorem 8.1 of \cite{S99} that $\nu=\nu_\sigma$, which further implies $\rho=\rho_\sigma$ on $S_n^+$.  Hence,  $\rho$ is permutation-invariant on $S_n^+$. We complete the proof.  \qed

{\it\bf Proof of Proposition \ref{prop:heter}}. Note that $\nu_i(\d x)=\lambda_ix^{-\alpha_i-1}\d x$ on $(0,\infty)$.  We first focus on (i).   It follows from Proposition \ref{th:main2} that the validity of  \eqref{Supm}  for all $(X_t^{(1)},\dots,X_t^{(n)})$ with $t>0$ and all $\boldsymbol\theta\in\Delta_n$ is equivalent to   $\sum_{\theta_i>0}\nu_i((\theta_i^{-1}x,\infty))\geq \nu_1((x,\infty))$ for $x>0$ and all $\boldsymbol\theta\in\Delta_n$, which is further equivalent to  $\sum_{\theta_i>0}\lambda_i\alpha_i^{-1}\theta_i^{\alpha_i} x^{-\alpha_i}\geq \lambda_1\alpha_1^{-1} x^{-\alpha_1}$ for $x>0$ and all $\boldsymbol\theta\in\Delta_n$. Note that for $i\neq 1$, $\lambda_i\alpha_i^{-1}x^{-\alpha_i}\geq \lambda_1\alpha_1^{-1} x^{-\alpha_1}$ is equivalent to $\lambda_i\alpha_i^{-1}x^{\alpha_1-\alpha_i}\geq \lambda_1\alpha_1^{-1}$ for all $x>0$, which implies $\alpha_i=\alpha_1$ and $\lambda_i\geq \lambda_1$. Conversely, if $\alpha_1=\cdots=\alpha_n$ and $\lambda_i\geq \lambda_1,~i\in[n]$, then $\sum_{\theta_i>0}\lambda_i\alpha_i^{-1}\theta_i^{\alpha_i} x^{-\alpha_i}\geq \lambda_1\alpha_1^{-1} x^{-\alpha_1}$ for all $\boldsymbol\theta\in\Delta_n$. Hence, (i) holds.   

For (ii), the ``if'' part follows directly from (ii) of Proposition \ref{Prop:2}. Next, we consider the ``only if'' part. By (i), we have $\alpha_1=\cdots=\alpha_n$ and $\lambda_i\geq \lambda_1,~i\in[n]$. It is trivial to show $\lambda_1=\cdots=\lambda_n$. We complete the proof. \qed

 {\it\bf Proof of Proposition \ref{Prop:ICP}}.
We first focus on (i). It follows from Proposition \ref{th:main2} that the validity of  \eqref{Supm}  for all $\mathbf X_t$ with $t>0$ and all $\boldsymbol\theta\in\Delta_n$ is equivalent to the following inequalities holding for  all $\boldsymbol\theta\in\Delta_n$: $\sum_{\theta_i>0}\lambda_i(1-F_{Y_i}(\theta_i^{-1}x))\geq \lambda_1(1-F_{Y_1}(x))$ for $x>0$ and  $\sum_{\theta_i>0}\lambda_iF_{Y_i}(\theta_i^{-1}x)\leq \lambda_1F_{Y_1}(x)$ for $x<0$. Let $f_i(x)=\lambda_iF_{Y_i}(x^{-1})$ for $x<0$. Then we have $\sum_{i=1}^nf_i(x_i)\leq f_1(x_1+\dots+x_n)$ for all $x_i<0$, which can be rewritten as $\sum_{i=2}^nf_i(x_i)\leq f_1(x_1+\dots+x_n)-f_1(x_1)$. Note that $f_i$ are decreasing and $\lim_{x\downarrow -\infty}f_i(x)=\lambda_iF_{Y_i}(0-)$. Hence, we have $\sum_{i=2}^nf_i(x_i)\leq \lim_{x_1\downarrow -\infty}(f_1(x_1+\dots+x_n)-f_1(x_1))=0$, which implies $F_{Y_2}(0-)=\dots=F_{Y_n}(0-)=0$.
Conversely, $F_{Y_2}(0-)=\dots=F_{Y_n}(0-)=0$ implies $\sum_{\theta_i>0}\lambda_iF_{Y_i}(\theta_i^{-1}x)\leq \lambda_1F_{Y_1}(x)$ for $x<0$ and all $\boldsymbol\theta\in\Delta_n$. Consequently, $\sum_{\theta_i>0}\lambda_iF_{Y_i}(\theta_i^{-1}x)\leq \lambda_1F_{Y_1}(x)$ for $x<0$ and all $\boldsymbol\theta\in\Delta_n$ is equivalent to $F_{Y_2}(0-)=\dots=F_{Y_n}(0-)=0$. Clearly, $\sum_{\theta_i>0}\lambda_i(1-F_{Y_i}(\theta_i^{-1}x))\geq \lambda_1(1-F_{Y_1}(x))$ for $x>0$ and $\boldsymbol\theta\in\Delta_n$ is equivalent to $\lambda_1(1-F_{Y_1}((x_1+\dots+x_n)^{-1}))\leq \sum_{i=1}^n\lambda_i(1-F_{Y_i}(x_i^{-1}))$ for all $x_i\in (0,\infty)$. We establish (i).

Next, we consider (ii). The ``if'' part is implied by (ii) of Proposition \ref{Prop:3}. For the ``only if'' part, note that $\boldsymbol e_i\preceq \boldsymbol e_j$ for all $i,j\in[n]$. Hence, we have $X_t^{(i)}\overset{d}{=}X_t^{(j)}$ for all $i,j\in [n]$ and all $t>0$. Hence, it boils down to the iid case and the conclusion follows from (ii) of Proposition \ref{Prop:3}. We complete the proof. \qed 

The following lemma  extends Theorem 2 of \cite{M25}, which will be used in the proofs later.
\begin{lemma}\label{le:extended}
\begin{enumerate}[(i)]
    \item For $m\geq 2$, suppose \eqref{Sup} holds for all $X_i,~i\in[m]$ and all $\boldsymbol \theta\in\Delta_n$, and $\phi:\R^m\to\R$ is an increasing convex function. If $X_i,~i\in[m]$ are independent, then  \eqref{Sup} holds for $\phi(X_1,\dots,X_m)$ for all $\boldsymbol \theta\in\Delta_n$.
    \item For $m\geq 2$, suppose $X_i\leq_{\mathrm{st}}Y_i,~i\in[m]$,   $X_i,i\in [m]$ are independent, and $Y_i,i\in [m]$ are independent. Moreover, $\phi:\R^m\to\R$ is an increasing  function. Then $\phi(X_1,\dots,X_m)\leq_{\mathrm{st}} \phi(Y_1,\dots,Y_m)$.
\end{enumerate}    
\end{lemma}
{\it \bf Proof of Lemma \ref{le:extended}}.  We first focus on (i). The idea of the proof is the same as that of Theorem 2 of \cite{M25}.  Let $X_i^{(j)},~j\in [n]$ be independent copies of $X_i$, and suppose $X_i^{(j)},~i\in [m], j\in [n]$ are mutually independent. Then  for $\boldsymbol \theta\in\Delta_n$, we have  $X_i\leq_{\mathrm{st}}\sum_{j=1}^n\theta_jX_i^{(j)}$. Since $\phi$ is increasing on each component and is also convex, then we have 
\begin{align*}
\phi(X_1,\dots,X_m)\leq_{\mathrm{st}}\phi\left(\sum_{j=1}^n\theta_jX_1^{(j)},\dots,\sum_{j=1}^n\theta_jX_m^{(j)}\right)\leq_{\mathrm{a.s.}}\sum_{j=1}^n\theta_j\phi(X_1^{(j)},\dots,X_m^{(j)}).
\end{align*}
We establish assertion (i). 
 Assertion (ii) follows directly from the preservation of the usual stochastic order under increasing functions and independence. This completes the proof.
 \qed 


{\it\bf Proof of Proposition \ref{Prop:ruin}}. Note that $\nu((x^{-1},\infty))$ is subadditive on $(0,\infty)$. Then it follows from (i) of Theorem \ref{th:main1} that $X_s\leq_{\mathrm{st}}\sum_{i=1}^n\theta_iX_s^{(i)}$ for all $s>0$ and all $\boldsymbol\theta\in \Delta_n$.  For $t>0$ and $m>1$, let $t_{k}^{(m)}=km^{-1}t$ for $k\in [m]$ and $\phi_m(x_1,\dots,x_m)=\max(x_1,x_1+x_2,\dots,x_1+\dots+x_m)$. Note that 
$\phi_m$ is an increasing function. Moreover, $X_{t_k^{(m)}}-X_{t_{k-1}^{(m)}}\leq_{\mathrm{st}}\sum_{i=1}^n\theta_i(X_{t_k^{(m)}}^{(i)}-X_{t_{k-1}^{(m)}}^{(i)}),~k\in [m]$ with $t_0^{(m)}=0$, $X_{t_k^{(m)}}-X_{t_{k-1}^{(m)}},~k\in [m]$ are independent, and $\sum_{i=1}^n\theta_i(X_{t_k^{(m)}}^{(i)}-X_{t_{k-1}^{(m)}}^{(i)}),~k\in [m]$ are independent. Then by (ii) of  Lemma \ref{le:extended}, we have $\max_{k=1}^m X_{t_{k}^{(m)}}\leq_{\mathrm{st}}\max_{k=1}^m \sum_{i=1}^n\theta_iX_{t_{k}^{(m)}}^{(i)}$. Letting $m\to\infty$ and using the fact that the sample paths are right-continuous, we have
\begin{align*}
\max_{0\leq s\leq t}X_s\leq_{\mathrm{st}} \max_{0\leq s\leq t}\sum_{i=1}^n\theta_iX_s^{(i)}
    \end{align*}
    for all $t>0$ and all $\boldsymbol\theta\in \Delta_n$. This shows the first inequality. The second one holds obviously.
\qed

{\it \bf Proof of Proposition \ref{Prop:SI}}. 
 If  $\nu((x^{-1},\infty))$ is subadditive on $(0,\infty)$, then \eqref{Sup} holds for all $X_t$ with $t>0$ and  all $\boldsymbol \theta\in\Delta_n$ due to (i) of Theorem \ref{th:main1}.  For $t>0$ and $m>1$, let $t_{k}^{(m)}=km^{-1}t$ for $k\in [m]$.
Then (i) of Lemma \ref{le:extended} implies \eqref{Sup} holds for $\sum_{k=1}^m f(t_{k}^{(m)},t)\left(X_{t_{k}^{(m)}}-X_{t_{k-1}^{(m)}}\right)$. Note  that $\lim_{m\to\infty}\sum_{k=1}^m f(t_{k}^{(m)},t)\left(X_{t_{k}^{(m)}}-X_{t_{k-1}^{(m)}}\right)=Z_t$ a.s. Hence, Theorem 2.3 of \cite{M25} implies that  \eqref{Sup} holds for all $Z_t$ with $t>0$ and  all $\boldsymbol \theta\in\Delta_n$.

If  $\nu((x^{-1},\infty))$ is concave on $(0,\infty)$, in light of (ii) of Theorem \ref{th:main1},  \eqref{Sub} holds for all $X_t$ with $t>0$ and  all $\boldsymbol \theta,\boldsymbol\eta\in\Delta_n$ satisfying $\boldsymbol\theta \preceq\boldsymbol\eta$. By Proposition 1 of \cite{CHSZ26}, \eqref{Sub} holds for $\sum_{k=1}^m f(t_{k}^{(m)},t)\left(X_{t_{k}^{(m)}}-X_{t_{k-1}^{(m)}}\right)$. Note also that $\lim_{m\to\infty}\sum_{k=1}^m f(t_{k}^{(m)},t)\left(X_{t_{k}^{(m)}}-X_{t_{k-1}^{(m)}}\right)=Z_t$ a.s. In light of  Proposition 1 of \cite{CHSZ26}, \eqref{Sub} holds for all $Z_t$ with $t>0$ and  all $\boldsymbol \theta,\boldsymbol\eta\in\Delta_n$ satisfying $\boldsymbol\theta \preceq\boldsymbol\eta$. This completes the proof.

\qed

{\it \bf Proof of Proposition \ref{Prop:integrated}}. 
 It follows from (i) of Theorem \ref{th:main1} that \eqref{Sup} holds for all $X_t$ with $t>0$ and  all $\boldsymbol \theta\in\Delta_n$.   For $t>0$ and $m>1$, let $t_{k}^{(m)}=km^{-1}t$ for $k\in [m]$ and $\phi_m(x_1,\dots,x_m)=m^{-1}t(g(x_1)+g(x_1+x_2)+\dots+g(x_1+\dots+x_m))$. Clearly, 
$$m^{-1}t\sum_{k=1}^mg\left(X_{t_k^{(m)}}\right)=\phi_m\left(X_{t_1^{(m)}}, X_{t_2^{(m)}}-X_{t_1^{(m)}},\dots,X_{t}-X_{t_{m-1}^{(m)}}\right).$$
Note also that $\phi_m$ is increasing and convex, and  \eqref{Sup} holds for all $X_{t_1^{(m)}}, X_{t_2^{(m)}}-X_{t_1^{(m)}},\dots,X_{t}-X_{t_{m-1}^{(m)}}$, and  all $\boldsymbol \theta\in\Delta_n$. Using  (i) of Lemma \ref{le:extended},  \eqref{Sup} holds for $m^{-1}t\sum_{k=1}^mg(X_{t_k^{(m)}})$ and  all $\boldsymbol \theta\in\Delta_n$. Moreover, using the fact that $\{X_t\}_{t\geq 0}$ has right-continuous sample paths, we have
$\lim_{m\to\infty}m^{-1}t\sum_{k=1}^mg(X_{t_k^{(m)}})=\int_0^tg(X_s)\d s$ a.s. By Theorem 2.3 of \cite{M25},  \eqref{Sup} holds for all $W_t$ with $t>0$ and  all $\boldsymbol \theta\in\Delta_n$.
\qed

{\it\bf Proof of Proposition \ref{ex:time-dependent-dominance}}.
We first verify that the compound Poisson process is well-defined.
Writing $w(z)=1-\sqrt{1-z}$, we obtain
$$
\frac{d}{dz}\log Q(z)
=\frac12\left\{
\frac{1}{1-w(z)^2}
+\frac{1}{(1-w(z))(2-w(z))}\right\}.
$$
Both rational functions of $w$ have nonnegative Taylor coefficients,
as does $w(z)$, with $w(0)=0$.
Thus $\lambda_k\ge0$ for all $k\ge1$; letting $z\uparrow1$ in
$\log Q(z)=\log Q(0)+\sum_{k\ge1}\lambda_k z^k$ gives
$\sum_{k\ge1}\lambda_k=\log4$. Hence, the compound Poisson process is well-defined. 
Moreover, $Q(z)$ can be rewritten as 
$$
Q(z)=\exp\left\{\sum_{k=1}^\infty\lambda_k(z^k-1)\right\},$$
which implies that  $Q(z)$ is the probability generating function of $X_1$ and for $t\geq 0$,
$$
\mathbb E[z^{X_t}]=Q(z)^t,\quad 0\le z\le1.
$$
Let $\psi(u)=-\log Q(e^{-u})$,
$\psi_\lambda(u)=\psi(\lambda u)+\psi((1-\lambda)u)$,
$a=3/2$, and $\kappa_\lambda=\sqrt\lambda+\sqrt{1-\lambda}$.
The L\'evy measure $\nu$ of $X$ has total mass $\Lambda=\sum_{k\ge1}\lambda_k=\log4$.
Direct expansion gives, as $u\downarrow0$,
$$
\psi(u)=a\sqrt u+O(u),\qquad
\psi'(u)=\frac{a}{2\sqrt u}+O(1),\qquad
-\psi''(u)\sim\frac{a}{4u^{3/2}}.
$$
Consequently,
$\psi_\lambda(u)=a\kappa_\lambda\sqrt u+O(u)$,
uniformly for $\lambda\in[\varepsilon,1-\varepsilon]$.
 Note that
$\psi(u)/u=\int_0^\infty e^{-ux}\overline\nu(x)\,dx\sim au^{-1/2}$ as $u\downarrow 0$, where $\overline\nu(x)=\nu((x,\infty))$. Using the Karamata Tauberian theorem (Theorem 1.7.1 in \cite{BGT87}) and monotone density theorem (Theorem 1.7.2 in \cite{BGT87}), we have, as $x\to\infty$, 
\begin{align}\label{Eq:asymptotics}
\int_0^x \overline{\nu}(t) \d t\sim 2C\sqrt{x}
\quad  \text{ and }\quad \overline\nu(x)\sim Cx^{-1/2}
 \quad \text{ with }
\quad C=\frac{a}{\sqrt\pi}.
\end{align}
Using integration by parts,
we further have, as $x\to\infty$, 
\begin{align}\label{Eq:asymptotics1}\int_{(0,x]}y\,\nu(dy)\sim C\sqrt x.
\end{align}
Moreover, for $u>0$ and $\lambda\in(0,1)$,
$$
\psi_\lambda(u)-\psi(u)
=\int_{(0,\infty)}
(1-e^{-\lambda uy})(1-e^{-(1-\lambda)uy})\,\nu(dy)>0.
$$

Fix $\varepsilon\in(0,1/2)$ and suppose that the first assertion
fails. Then there exist $t_j\to\infty$,
$\lambda_j\in[\varepsilon,1-\varepsilon]$, and $x_j\ge0$ such that,
with $Y_j=\lambda_jX_{t_j}^{(1)}
+(1-\lambda_j)X_{t_j}^{(2)}$,
$$
\mathbb P(Y_j\le x_j)>\mathbb P(X_{t_j}\le x_j).
$$
Passing to a subsequence, assume that
$\lambda_j\to\lambda\in[\varepsilon,1-\varepsilon]$.
We obtain a contradiction in each of the following four cases,
which exhaust all possibilities after further subsequence selection.

\emph{Case 1: $x_j/t_j\to v\in[0,\infty)$.}
If $v=0$, choose $u>0$ such that $\psi_\lambda(u)>3\Lambda/2$;
this is possible since $\psi_\lambda(u)\to2\Lambda$ as
$u\to\infty$. Since $\lambda_j\to\lambda$, there exists $j_0>1$ such that $\psi_{\lambda_j}(u)>3\Lambda/2$ for $j\geq j_0$.  Using Markov's inequality and the distribution of $X_{t_j}$, we have, for sufficiently large $j\geq j_0$,
$$
\mathbb P(Y_j\le x_j)\leq \E(e^{-u(Y_j-x_j)}\id_{\{Y_j\leq x_j\}})\leq e^{u x_j}\E(e^{-uY_j})
=e^{ux_j-t_j\psi_{\lambda_j}(u)}
<e^{-\Lambda t_j}
\le\mathbb P(X_{t_j}\le x_j),
$$
which contradicts $\mathbb P(Y_j\le x_j)>\mathbb P(X_{t_j}\le x_j)$.
Suppose next that $v>0$. Clearly, $\psi'(0+)=\infty$.  Note that $\psi'(u)=\int_0^\infty x e^{-ux}\nu(\d x)$. Dominated convergence theorem implies $\lim_{u\to\infty}\psi'(u)=0$. Using the continuity and strict monotonicity of $\psi'(u)$ over $(0,\infty)$,
there exists a unique $u_v>0$ satisfying $\psi'(u_v)=v$.
Set $I(v)=\psi(u_v)-u_vv$.
Next, we show that the exponential tilting gives
$$
\liminf_{j\to\infty}\frac1{t_j}\log\mathbb P(X_{t_j}\le x_j)
\geq-I(v).
$$
Define the exponential tilted probability measure by
$$
\frac{d\mathbb P_{t,u}}{d\mathbb P}
=e^{-uX_t+t\psi(u)}.$$
Under $\mathbb P_{t,u}$,
$$
\mathbb E_{t,u}[X_t]=t\psi'(u),
\quad
\operatorname{Var}_{t,u}(X_t)=-t\psi''(u).
$$
Taking $u>u_v$, then it follows that $\psi'(u)<v$. Choosing $0<\delta<v-\psi'(u)$, we have $\psi'(u)+\delta<v$. Since $x_j/t_j\to v$ as $j\to\infty$, there exists $j_1>0$ such that $\psi'(u)+\delta<x_j/t_j$ for $j\geq j_1$. Let $A_j=\{|X_{t_j}/t_j- 
\psi'(u)|<\delta\}$. Clearly, $A_j\subseteq \{X_{t_j}\leq x_j\}$ for $j\geq j_1$. Hence, we have
\begin{align*}
    \mathbb P(X_{t_j}\le x_j)\geq \mathbb P(A_j)=\p_{t_j,u}\left(e^{uX_{t_j}-t_j\psi(u)}\id_{A_j}\right)\geq e^{t_j(u(\psi'(u)-\delta)-\psi(u))}\p_{t_j,u}(A_j).
\end{align*}
Moreover, Chebyshev's inequality implies that
$\p_{t_j,u}(A_j)\geq 1-\frac{-\psi''(u)}{t_j\delta^2}\to 1 $
as $j\to\infty$.
Therefore,  
$$
\liminf_{j\to\infty}\frac1{t_j}\log\mathbb P(X_{t_j}\le x_j)
\geq u(\psi'(u)-\delta)-\psi(u).
$$
The lower bound is obtained by letting $\delta\downarrow 0$ and $u\downarrow u_v$.
On the other hand, using Markov's inequality, we have
$$\p(Y_j\le x_j)=\p(e^{-u_vY_j}\geq e^{-u_vx_j})\leq e^{u_vx_j}\E(e^{-u_vY_j})=e^{u_vx_j-t_j\psi_{\lambda_j}(u_v)}.$$
Hence, using the fact that $\psi_\lambda(u)>\psi(u)$ for $u>0$ and $0<\lambda<1$, we have
$$
\limsup_{j\to\infty}\frac1{t_j}
\log\mathbb P(Y_j\le x_j)\leq \limsup_{j\to\infty}\left(u_v\frac{x_j}{t_j}-\psi_{\lambda_j}(u_v)\right)=u_vv-\psi_\lambda(u_v)
<u_vv-\psi(u_v)=-I(v),
$$
which contradicts $\mathbb P(Y_j\le x_j)>\mathbb P(X_{t_j}\le x_j)$.

\emph{Case 2: $x_j/t_j\to\infty$ and $x_j/t_j^2\to0$.}
We first show that
$$
\log\mathbb P(X_{t_j}\le x_j)
\geq -\left(\frac{a^2}{4}+o(1)\right)\frac{t_j^2}{x_j}.
$$
Fix $\delta\in(0,1/2)$ and take
$u_j=(at_j/(2(1-\delta)x_j))^2$.
Under $\mathbb P_{t_j,u_j}$, we have, 
$$
\mathbb E_{t_j,u_j}[X_{t_j}]=t_j\psi'(u_j)
=(1-\delta)x_j+O(t_j),
\qquad
\operatorname{Var}_{t_j,u_j}(X_{t_j})=-t_j\psi''(u_j)
=O(x_j^3/t_j^2)=o(x_j^2).
$$
Let $A_j'=\{(1-2\delta)x_j\le X_{t_j}\le x_j\}$. 
Then Chebyshev's inequality implies that
$\p_{t_j,u_j}(A_j')\geq \p_{t_j,u_j}(\{|X_{t_j}-t_j\psi'(u_j)|\leq (x_j\delta)/2\})\geq 1-\frac{-4t_j\psi''(u_j)}{x_j^2\delta^2}\to 1 $
as $j\to\infty$. Moreover, direct computation gives
\begin{align*}
    \mathbb P(X_{t_j}\le x_j)\geq \E_{t_j,u_j}\left(e^{u_jX_{t_j}-t_j\psi(u_j)}\id_{A_j'}\right)\geq e^{u_jx_j(1-2\delta)-t_j\psi(u_j)}\p_{t_j,u_j}(A_j').
\end{align*}
Hence, we have 
$$
\begin{aligned}
\log\mathbb P(X_{t_j}\le x_j)
&\ge u_j(1-2\delta)x_j-t_j\psi(u_j)+o(1)\\
&=-\left(\frac{a^2}{4(1-\delta)^2}+o(1)\right)
\frac{t_j^2}{x_j}.
\end{aligned}
$$
Letting $\delta\downarrow0$ proves the claimed lower bound.

Moreover, using Markov's inequality with $u_j=(a\kappa_{\lambda_j}t_j/(2x_j))^2$, we have
$$\p(Y_j\le x_j)=\p(e^{-u_jY_j}\geq e^{-u_jx_j})\leq e^{u_jx_j}\E(e^{-u_jY_j})=e^{u_jx_j-t_j\psi_{\lambda_j}(u_j)},$$
which implies 
$$
\log\mathbb P(Y_j\le x_j)
\le-\left(\frac{a^2\kappa_{\lambda_j}^2}{4}+o(1)\right)
\frac{t_j^2}{x_j}.
$$
Since $\kappa_{\lambda_j}>1$, this again yields a contradiction.

\emph{Case 3: $x_j/t_j^2\to v\in(0,\infty)$.}
Using the expansion of $\psi$ around zero, we have for $u>0$
$$
\lim_{j\to\infty}\E\left(e^{-ut_j^{-2}X_{t_j}}\right)=e^{-t_j\psi(t_j^{-2}u)}=e^{-a\sqrt u}, \quad \lim_{j\to\infty}\E\left(e^{-ut_j^{-2}Y_j}\right)=e^{-t_j\psi_{\lambda_j}(t_j^{-2}u)}=e^{-a\ \kappa_\lambda\sqrt u},
$$
which implies that,
as  $j\to\infty$,
$$
\frac{X_{t_j}}{t_j^2}\ \rightarrow\ Z,
\qquad
\frac{Y_j}{t_j^2}\ \rightarrow\ \kappa_\lambda^2Z
$$
in distribution, where $\mathbb E[e^{-uZ}]=e^{-a\sqrt u}$.
Note that the random variable $Z$ has a strictly positive density
$$
f_Z(x)=\frac{a}{2\sqrt\pi}x^{-3/2}e^{-a^2/(4x)},
\qquad x>0.
$$
Consequently, as $j\to\infty$,
$$
\mathbb P(X_{t_j}\le x_j)-\mathbb P(Y_j\le x_j)
\longrightarrow F_Z(v)-F_Z(v/\kappa_\lambda^2)>0,
$$
another contradiction.

\emph{Case 4: $x_j/t_j^2\to\infty$.}
We first show that 
$$
\mathbb P(X_t>x)\sim Ct/\sqrt x
\qquad\text{whenever }t\to\infty
\text{ and }x/t^2\to\infty.
$$
Fix $\eta\in (0,1/2)$. For the compound Poisson process, let  $A_1$ represent the event that at least one jump larger than  $(1-\eta)x$  over  $[0,t]$,
$A_2$ represent the event that at least two jumps larger than $\eta^4x$  over  $[0,t]$ and $A_3=\{ S_t>\eta x\}$, where $S_t$ represents the cumulative jumps smaller than $\eta^4 x$ over $[0,t]$.
Clearly, $\{X_t>x\}\subseteq A_1\cup A_2\cup A_3$. Using the Poisson random measure representation of Poisson processes, we have 
$\p(A_1)=1-e^{-t\overline\nu((1-\eta)x)}\leq t\overline\nu((1-\eta)x)$ and $\p(A_2)=1-e^{-t\overline\nu(\eta^4x)}(1+t\overline\nu(\eta^4x))\leq (t\overline\nu(\eta^4 x))^2/2$. Applying Markov's inequality, we have
$$\p(A_3)\leq \frac{\E(S_t)}{\eta x}=\frac{t}{\eta x}\int_{(0,\eta^4x]}y\,\nu(dy).$$
Summarizing the above inequalities, we have 
$$
\begin{aligned}
\mathbb P(X_t>x)\le{}&
t\overline\nu((1-\eta)x)
+\frac12\bigl(t\overline\nu(\eta^4x)\bigr)^2
+\frac{t}{\eta x}\int_{(0,\eta^4x]}y\,\nu(dy).
\end{aligned}
$$
Using the asymptotics in \eqref{Eq:asymptotics} and \eqref{Eq:asymptotics1}, we have  $\p(X_t>x)\leq  C((1-\eta)^{-1/2}+\eta+o(1))t/\sqrt x$.
Letting $\eta\downarrow0$ proves the upper bound. Using the event of a jump exceeding $x$ gives the lower
bound $\mathbb P(X_t>x)\geq 1-e^{-t\overline\nu(x)}\sim t\overline\nu(x)\sim Ct/\sqrt x$. We establish the claimed asymptotics.

Applying it at $x_j/\lambda_j$ and $x_j/(1-\lambda_j)$,
and using independence, we obtain
$$
\begin{aligned}
\mathbb P(Y_j>x_j)
&\ge\mathbb P\bigl(
\{\lambda_jX_{t_j}^{(1)}>x_j\}
\cup\{(1-\lambda_j)X_{t_j}^{(2)}>x_j\}\bigr)\\
&=(\kappa_{\lambda_j}+o(1))\frac{Ct_j}{\sqrt{x_j}}
>\mathbb P(X_{t_j}>x_j)
\end{aligned}
$$
for all sufficiently large $j$, giving the final contradiction.
This proves the first assertion.

For the second assertion, let $s=1/10$ and
$x_\lambda=\max\{\lambda,1-\lambda\}<1$.
Recall that $X_s$ is a compound Poisson random variable with parameter $\Lambda=\log 4$ and  $\nu(\{1\})=\frac{d}{dz}\log Q(z)\big|_{z=0}=3/4$. Direct computation gives
$\mathbb P(X_s=0)=4^{-s}$ and
$\mathbb P(X_s=1)=(3s/4)4^{-s}$.
The outcomes $(0,0)$, $(1,0)$, and $(0,1)$ therefore give
$$
\begin{aligned}
\mathbb P\bigl(\lambda X_s^{(1)}+(1-\lambda)X_s^{(2)}
\le x_\lambda\bigr)
&\ge4^{-2s}\left(1+\frac{3s}{2}\right)\\
&>4^{-s}=\mathbb P(X_s\le x_\lambda),
\end{aligned}
$$
where $4^{1/10}<23/20$ proves the strict inequality. This completes the proof.
\qed

\subsection{An additional result on multivariate $\alpha$-stable distributions}\label{Appendix:additional}

In this subsection, we consider the multivariate $\alpha$-stable distributions without the restriction that the L\'{e}vy measure is concentrated on the positive orthant.
For $\boldsymbol\theta\in\Delta_n$, let $A^+(\boldsymbol\theta)=\int_{S_n}\left(\sum_{i=1}^n\theta_i\xi_i\right)_+^{\alpha}\rho(\d \boldsymbol \xi)$ and $A^-(\boldsymbol\theta)=\int_{S_n}\left(\sum_{i=1}^n\theta_i\xi_i\right)_-^{\alpha}\rho(\d \boldsymbol \xi)$.
\begin{proposition}\label{Prop:additinal} Let $\{\mathbf X_t\}_{t\geq 0}$ be an $n$-dimensional $\alpha$-stable process with  $\alpha\in (0,1)$ and generating triplet $(0, \nu, \boldsymbol\gamma_0)$ with $\nu$ given by \eqref{stablev}. Then we have the following conclusions.
\begin{enumerate}[(i)]
\item  Fix $\boldsymbol\theta\in\Delta_n$. Then the stochastic dominance in \eqref{Supm} holds for all $\mathbf X_t$ with $t>0$ if and only if $A^+(\boldsymbol\theta)\geq A^+(\boldsymbol e_1)$ and $A^-(\boldsymbol\theta)\leq A^-(\boldsymbol e_1)$.
\item Fix $\boldsymbol\theta,\boldsymbol\eta\in \Delta_n$ satisfying $\boldsymbol\theta \preceq\boldsymbol\eta$. Then the  stochastic dominance in \eqref{Subm} holds for all $\mathbf X_t$ with $t>0$ if and only if $A^+(\boldsymbol\theta)\geq A^+(\boldsymbol\eta)$ and $A^-(\boldsymbol\theta)\leq A^-(\boldsymbol\eta)$.
\end{enumerate}
\end{proposition}
  In particular, part (ii) of Proposition \ref{Prop:additinal} shows that \eqref{Subm} holds for all $\mathbf X_t$ with $t>0$ and all $\boldsymbol\theta,\boldsymbol\eta\in \Delta_n$ satisfying $\boldsymbol\theta \preceq\boldsymbol\eta$ if and only if $A^+$ is Schur-concave and $A^-$ is Schur-convex. We refer to \cite{MOA11} for the definitions, characterizations and applications of Schur-concavity and Schur-convexity. In general,  it is difficult to obtain a further simplification of the equivalent conditions in Proposition \ref{Prop:additinal}. 

{\it \bf Proof of Proposition \ref{Prop:additinal}}. For $\boldsymbol\theta\in\Delta_n$, let $B_{u,\boldsymbol\theta}^+=\left\{\mathbf x\in\R^n: \sum_{i=1}^n\theta_ix_i>u\right\}$ and $B_{u,\boldsymbol\theta}^-=\left\{\mathbf x\in\R^n: \sum_{i=1}^n\theta_ix_i<u\right\}$.
    Then a direct computation shows that for $u>0$
    \begin{align}\label{Eq:alphap}
        \nu(B_{u,\boldsymbol\theta}^+)&=\int_{S_n}\int_0^\infty r^{-1-\alpha}\id_{B_{u,\boldsymbol\theta}^+}(r\boldsymbol \xi)\d r\rho(\d \boldsymbol \xi)\nonumber\\
        &=\int_{S_n}\int_{u(\sum_{i=1}^n\theta_i\xi_i)_+^{-1}}^\infty r^{-1-\alpha}\d r\rho(\d \boldsymbol \xi)\nonumber\\
        &=   \frac{ u^{-\alpha}}\alpha\int_{S_n}\left(\sum_{i=1}^n\theta_i\xi_i\right)_+^{\alpha}\rho(\d \boldsymbol \xi),
    \end{align}
    and for $u<0$,
\begin{align}\label{Eq:alpham}
        \nu(B_{u,\boldsymbol\theta}^-)&=\int_{S_n}\int_0^\infty r^{-1-\alpha}\id_{B_{u,\boldsymbol\theta}^-}(r\boldsymbol \xi)\d r\rho(\d \boldsymbol \xi)\nonumber\\
        &=\int_{S_n}\int_{|u|(\sum_{i=1}^n\theta_i\xi_i)_-^{-1}}^\infty r^{-1-\alpha}\d r\rho(\d \boldsymbol \xi)\nonumber\\
        &=   \frac{ |u|^{-\alpha}}\alpha\int_{S_n}\left(\sum_{i=1}^n\theta_i\xi_i\right)_-^{\alpha}\rho(\d \boldsymbol \xi).
    \end{align}
 
    Applying Proposition \ref{th:main2} and using \eqref{Eq:alphap} and \eqref{Eq:alpham}, for any fixed $\boldsymbol\theta\in\Delta_n$, we have \eqref{Supm} holds for all $\mathbf X_t$ with $t>0$ if and only if  
    $A^+(\boldsymbol\theta)\geq A^+(\boldsymbol e_1)$ and $A^-(\boldsymbol\theta)\leq A^-(\boldsymbol e_1)$. This proves (i).
    
    Similarly, for any fixed $\boldsymbol\theta,\boldsymbol\eta\in \Delta_n$ satisfying $\boldsymbol\theta \preceq\boldsymbol\eta$,  Proposition \ref{th:main2} together with \eqref{Eq:alphap} and \eqref{Eq:alpham} implies that  the stochastic dominance in \eqref{Subm} holds for all $\mathbf X_t$ with $t>0$ if and only if $A^+(\boldsymbol\theta)\geq A^+(\boldsymbol\eta)$ and $A^-(\boldsymbol\theta)\leq A^-(\boldsymbol\eta)$. This proves (ii). We complete the proof.
  \qed

\end{document}